\documentclass[11pt,oneside,a4paper]{article}
\usepackage{blindtext}
\usepackage{amsmath,amsfonts,amssymb,amsthm,eucal}
\usepackage{microtype}
\usepackage[colorlinks, citecolor=blue, linkcolor=blue, urlcolor=Maroon, filecolor=Maroon, linktocpage=true]{hyperref}
\usepackage[usenames,dvipsnames,svgnames,table]{xcolor}
\usepackage[T1]{fontenc}
\usepackage{lmodern}
\usepackage{enumerate}
\usepackage{bbm}
\usepackage{tikz}
\usepackage{tikz-cd}
\usepackage[font=footnotesize,labelfont=bf]{caption}
\usepackage{subcaption}
\usepackage{framed}
\usepackage[nottoc,notlot,notlof]{tocbibind}
\usepackage{tocloft}
\usepackage{morefloats}
\usepackage{float}
\usepackage[nottoc,notlot,notlof]{tocbibind}
\usepackage{tocloft}
\usepackage[backgroundcolor=LightBlue,linecolor=LightBlue]{todonotes}
\usepackage{hhline}
\usepackage{varwidth}
\usepackage{tikz,pgfplots}

\newcommand{\D}{\,\mathrm{d}}

\theoremstyle{definition}
\newtheorem{Th}{Theorem}[section]

\newtheorem{Prop}[Th]{Proposition}

\newtheorem{Lem}[Th]{Lemma}
\newtheorem{Def}[Th]{Definition}

\theoremstyle{definition}
\newtheorem{rem}[Th]{Remark}
\usepackage[a4paper,
	hcentering,
	text={453pt,686pt},
]{geometry}
\allowdisplaybreaks
\catcode`@=11
\@addtoreset{equation}{section}
\catcode`@=12
\begin{document}
\begin{center}\noindent
\textbf{\Large Special homogeneous surfaces}\\[2em]
{\fontseries{m}\fontfamily{cmss}\selectfont \large David Lindemann${}^\dagger$, Andrew Swann${}^{\dagger,\ddagger}$}\\[1em] 
{\small
${}^\dagger$Department of Mathematics and ${}^\ddagger$DIGIT, Aarhus University\\
Ny Munkegade 118, Bldg 1530, DK-8000 Aarhus C, Denmark\\
\texttt{david.lindemann@math.au.dk, swann@math.au.dk}
}
\end{center}
\vspace{1em}
\begin{abstract}
\noindent
We classify hyperbolic polynomials in two real variables that admit a transitive action on some component of their hyperbolic level sets. Such surfaces are called special homogeneous surfaces, and they are equipped with a natural Riemannian metric obtained by restricting the negative Hessian of their defining polynomial. Independent of the degree of the polynomials, there exist a finite number of special homogeneous surfaces. They are either flat, or have constant negative curvature.
\end{abstract}
\textbf{Keywords:} affine differential geometry, centro-affine surfaces, special real geometry, real algebraic surfaces, projective curves, homogeneous spaces\\
\textbf{MSC classification:} 53A15 (primary), 51N35, 14M17, 53C30, 53C26 (secondary)
\tableofcontents
\paragraph{Acknowledgements}\textcolor{white}{.}\\
David Lindemann is funded by a Walter Benjamin postdoc fellowship kindly granted by the \textit{German Research Foundation} (DFG) with title ``Applications of special real geometry in K\"ahler geometry''.
\section{Introduction and main results}
	The aim of this work is to classify all special homogeneous surfaces. A special homogeneous surface is a two-dimensional homogeneous space $\mathcal{H}$ that is contained in the level set $\{h=1\}$ of a homogeneous polynomial $h:\mathbb{R}^3\to\mathbb{R}$ of degree $\tau$ at least three, such that at any point $p\in\mathcal{H}$ the negative Hessian of $h$ at $p$, $-\partial^2h_p$, has Minkowski signature. Real homogeneous polynomials $h$ admitting such a point $p$ are called hyperbolic, and $p$ is called a hyperbolic point of $h$. In general dimension, hypersurfaces that are contained in
		\begin{equation*}
			\mathrm{hyp}_1(h):=\{h=1\}\cap\mathrm{hyp}(h),
		\end{equation*}
	where $\mathrm{hyp}(h)$ denotes the cone of hyperbolic points of $h$, are called generalised projective special real (GPSR) manifolds for $\tau\geq 4$, or simply projective special real (PSR) manifolds if $h$ is cubic \cite{L2}. The restriction of $-\frac{1}{\tau}\partial^2h$ to the tangent bundle of a (G)PSR manifold $\mathcal{H}\subset\{h=1\}$ is a Riemannian metric, which follows from the hyperbolicity of points in $\mathcal{H}$ and the homogeneity of $h$. For more details and an explanation of the scaling factor $\frac{1}{\tau}$ see \cite[Prop.\,1.3]{CNS}. The study of (G)PSR manifolds is motivated by both pure mathematics and theoretical physics, more specifically supergravity.
	
	When studying compact K\"ahler $\tau$-folds $X$, the volume function on the real $(1,1)$-cohomology,
		\begin{equation*}
			h:H^{1,1}(X,\mathbb{R})\to\mathbb{R},\quad [\omega]\mapsto\int\limits_X \omega^\tau,
		\end{equation*}
	is one tool in order to understand the geometry of the K\"ahler cone $\mathcal{K}\subset H^{1,1}(X,\mathbb{R})$ of $X$. The real homogeneous polynomial $h$ of degree $\tau$ has the property that every point $p$ in $\mathcal{K}$ is a hyperbolic point. The curvature of the hypersurfaces $\{h=1\}\cap\mathcal{K}$ has been studied in \cite{W} for certain Calabi-Yau $3$-folds. In general, it is not well understood which homogeneous polynomials $h$ can be realised in this manner. A reasonable approach is to try and find fitting K\"ahler manifolds for the subset of homogeneous polynomials that admit a transitive automorphism group. From this perspective, our work extends \cite{L5} in which special homogeneous curves have been classified.
	
	In 4+1-dimensional supergravity, level sets of hyperbolic cubic polynomials play the role of the target space of the vector multiplets. Such hypersurfaces are called projective special real (PSR) manifolds. Using the supergravity $r$- and $c$-map constructions \cite{CHM,DV,MS}, one can obtain explicit examples of projective special K\"ahler and quaternion K\"ahler manifolds \cite{CDJL}. In \cite{DV} homogeneous PSR manifolds have been completely classified, and their $r$- and $q=c\circ r$-map images are studied. For a general classification of all homogeneous polynomials, independent of the number of variables, that contain a homogeneous space in their level set $\{h=1\}$, such that every point in that space is hyperbolic, we currently lack the necessary technical tools. In our main Theorem \ref{thm_main} of this work we successfully solve this problem in dimension two.
	
	\begin{Th}\label{thm_main}
		Let $\mathcal{H}$ be a special homogeneous surface. Then $\mathcal{H}$ is contained in the level set $\{h=1\}$ of precisely one of the following hyperbolic polynomials $h:\mathbb{R}^3\to\mathbb{R}$ of degree $\tau\geq 3$:
			\begin{enumerate}[(i)]
				\item \label{thm_main_i} For $\tau$ even: $h=\left(x^2-y^2-z^2\right)^{\frac{\tau}{2}}$, $\mathrm{hyp}_1(h)$ has two isometric connected components, each homothetic to $\mathrm{SO}(2,1)^{+}/\mathrm{SO}(2)$. In particular, every hyperbolic point of $h$ has one-dimensional stabiliser. The automorphism group of $h$ is given by $\mathrm{O}(2,1)$.
				\item \label{thm_main_ii} For any $\tau$: $h=(x+z)^{\tau-2k}\left(x^2-y^2-z^2\right)^k$ for precisely one $k$ with $1\leq k<\frac{\tau}{2}$, $\mathrm{hyp}_1(h)$ has one connected component if $\tau$ is odd, and two isometric connected components for $\tau$ even. The identity component of the automorphism group $G^h_0$ of $h$ is isomorphic to the two-dimensional affine Lie group, the full group is $G^h\cong G^h_0\ltimes \mathbb{Z}_2$ for $\tau$ odd and $G^h\cong (G^h_0\times\mathbb{Z}_2)\ltimes\mathbb{Z}_2$ for $\tau$ even. In any of these cases the stabiliser of hyperbolic points is isomorphic to $\mathbb{Z}_2$.
				\item \label{thm_main_iii} For any $\tau$: $h=x^a y^b z^c$ for precisely one triple $(a,b,c)$ with $1\leq a\leq b\leq c\leq \tau-2$ and $a+b+c=\tau$. The set $\mathrm{hyp}_1(h)$ has $8$ isometric connected components for $a$, $b$, and $c$ even, and $4$ isometric connected components otherwise. The automorphism groups $G^h$ for different values of $a,b,c$ are listed in \eqref{eqn_Gh_flat_case_even} and \eqref{eqn_Gh_flat_case_odd}.
			\end{enumerate}
	\end{Th}
	\noindent
	\begin{minipage}{0.32\textwidth}
		\begin{figure}[H]%
			\centering%
			\includegraphics[scale=0.2]{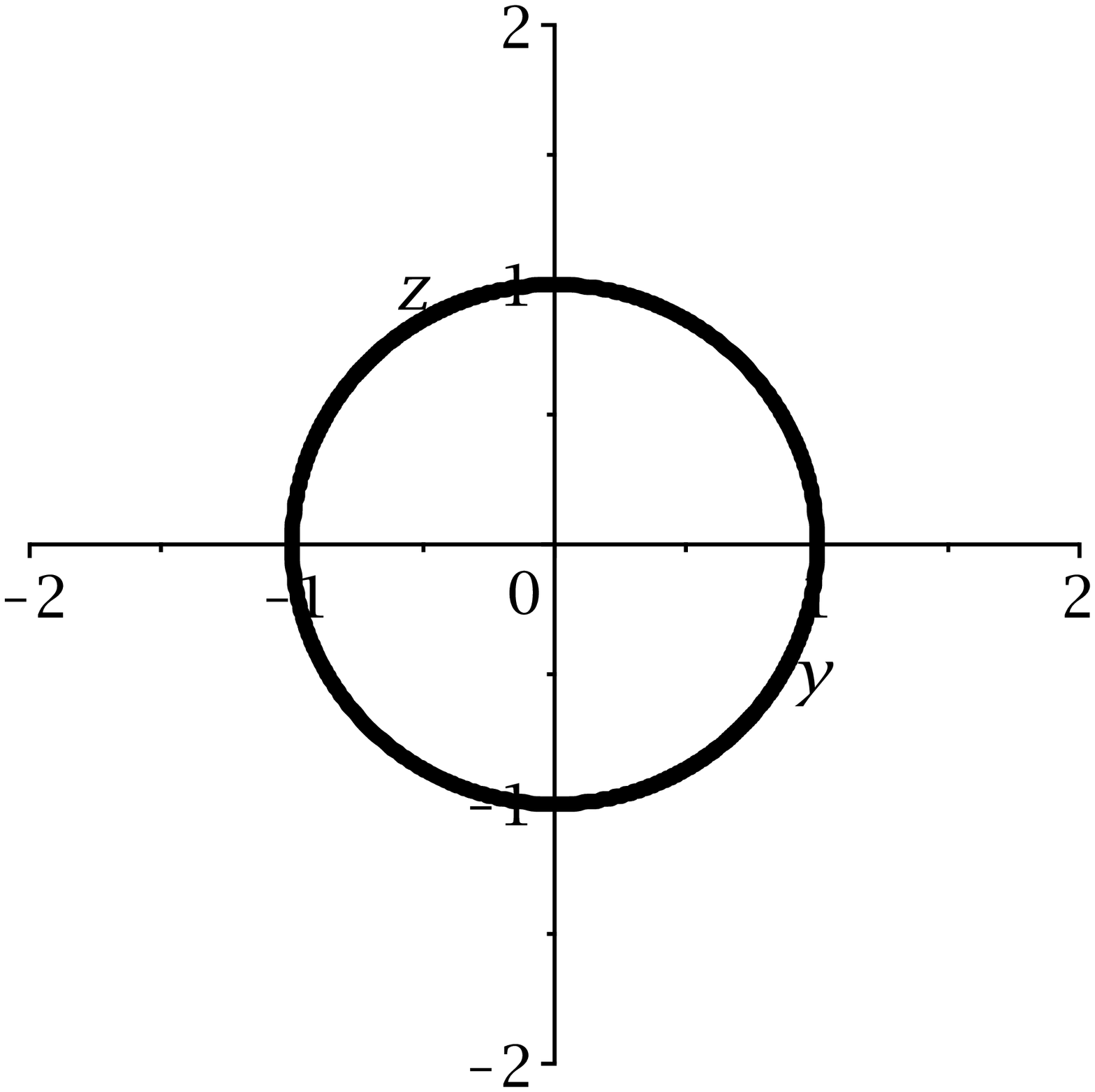}%
			\caption{The zero set $\{h=0\}$ of Thm. \ref{thm_main} \eqref{thm_main_i} on $\{x=1\}$.
			}
		\end{figure}
	\end{minipage}
	\begin{centering}
		\begin{minipage}{0.32\textwidth}
			\begin{figure}[H]%
				\centering%
				\includegraphics[scale=0.2]{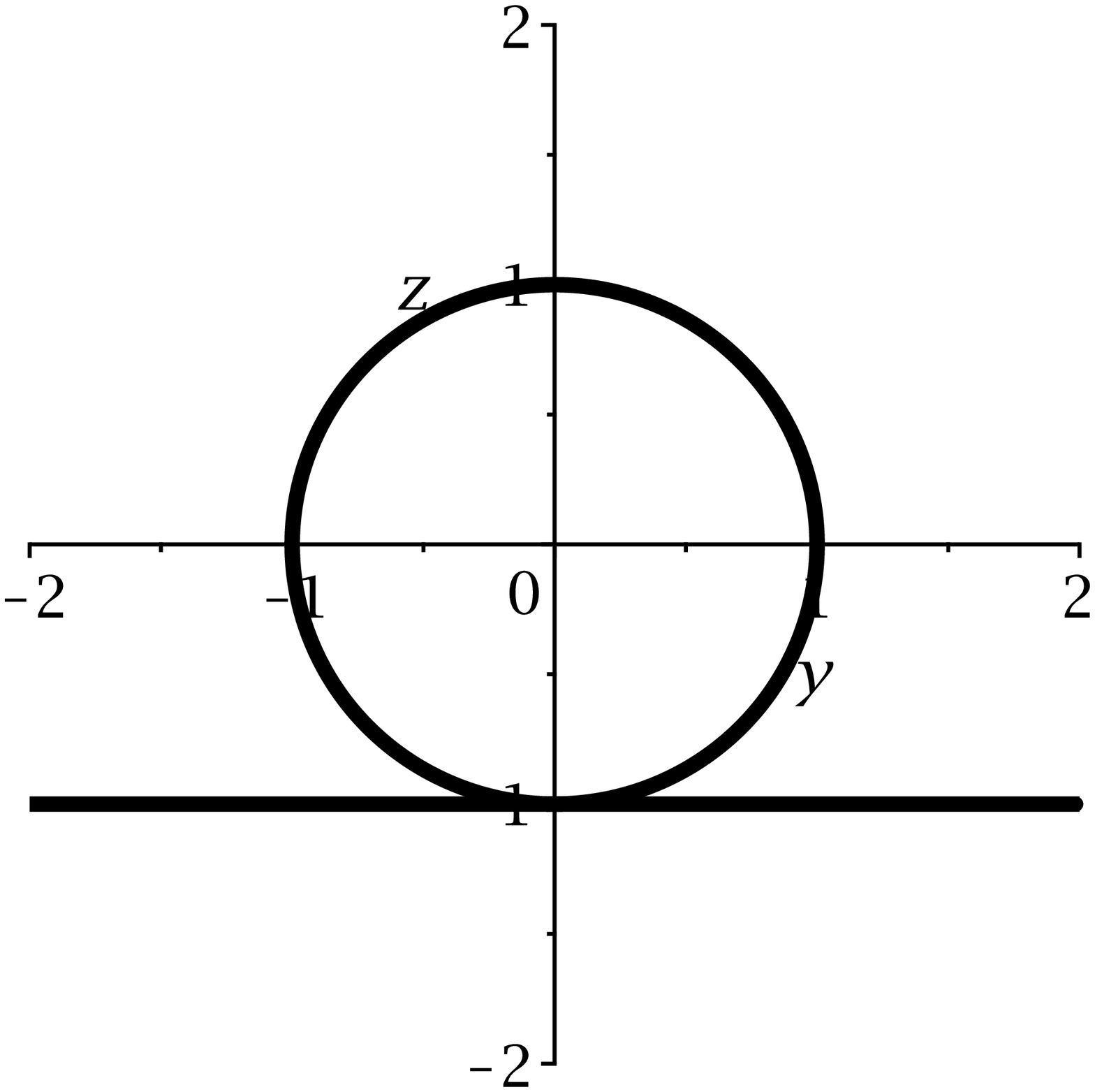}%
				\caption{The zero set $\{h=0\}$ of Thm. \ref{thm_main} \eqref{thm_main_ii} on $\{x=1\}$.
				}
			\end{figure}
		\end{minipage}
	\end{centering}
	\hfill
	\begin{minipage}{0.32\textwidth}
		\begin{figure}[H]%
			\centering%
			\includegraphics[scale=0.2]{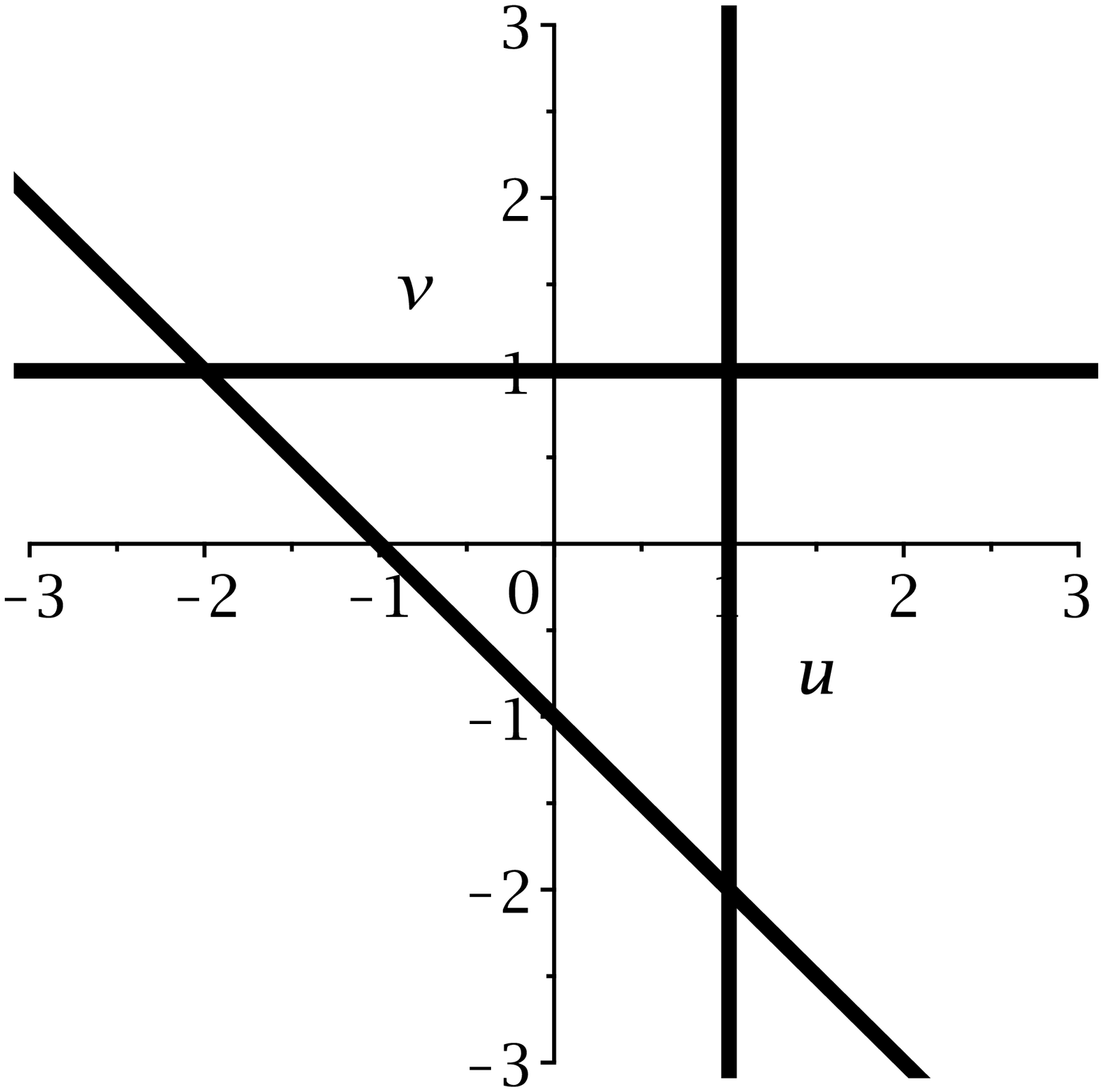}%
			\caption{The zero set $\{h=0\}$ of Thm. \ref{thm_main} \eqref{thm_main_iii} on $\{x=1+u+v,\ y=1-u,\ z=1-v\}$.}
		\end{figure}
	\end{minipage}
	
	A question that naturally arises when studying homogeneous polynomials is to ask whether the corresponding projective varieties are singular. In the setting of hyperbolic polynomials $h$, we might run into the situation where the connected components of $\mathrm{hyp}(h)$ are not pairwise equivalent, and one has a singular point in its boundary while the other does not. Thus we say that a connected component of $\mathrm{hyp}(h)$, or a connected component $\mathcal{H}\subset\mathrm{hyp}_1(h)$ depending on the context, is singular at infinity if its boundary contains a singular point of $h$. That means that there exists $p\in\partial\left(\mathbb{R}_{>0}\cdot\mathcal{H}\right)$, such that $\D h_p=0$. Note that singular at infinity implies $(h=0)$ being singular as a real projective variety. In the case of special homogeneous surfaces, we obtain the following result.
	
	\begin{Prop}\label{prop_shs_sing_at_inf}
		Special homogeneous surfaces are singular at infinity.
	\end{Prop}
	
		Note that $\mathcal{H}\subset\{h=1\}$ being singular at infinity is a stronger condition than just requiring that $(h=0)$ is a singular algebraic curve. While in the case of homogeneous (G)PSR surfaces, that is special homogeneous surfaces, we have shown that it is in fact equivalent, this might not hold for homogeneous (G)PSR manifolds in higher dimension.
		
		We also study the curvature of special homogeneous surfaces. This is motivated by \cite{W} where the scalar curvature of hyperbolic level sets in the K\"ahler cone of intersection Calabi-Yau $3$-folds were studied, and furthermore by the results of \cite{L3}. In the latter, the asymptotic behaviour of PSR manifolds that are closed in the ambient space are studied, and one particular result \cite[Thm.\,1.15]{L3} is that asymptotically, closed PSR manifolds that are not singular at infinity behave asymptotically as a metric space like precisely one of the homogeneous PSR manifolds. While the implications for the asymptotic curvature behaviour for PSR manifolds, closed or not, is not completely understood yet, this result points to the importance of understanding the curvature of homogeneous (G)PSR manifolds in order to describe asymptotics of (G)PSR manifolds in general. The two-dimensional case that we are studying in this work represents the simplest non-trivial case in that regard.

	\begin{Prop}\label{prop_scalar_curvature}
		The scalar curvatures $S$ of the special homogeneous homogeneous surfaces $\mathcal{H}\subset\{h=1\}$ in Theorem \ref{thm_main} with respect to the centro-affine metric $g=-\frac{1}{\tau}\partial^2h|_{T\mathcal{H}\times T\mathcal{H}}$ is given by
			\begin{equation*}
				\begin{array}{ll}
					\text{Thm. \ref{thm_main} \eqref{thm_main_i}:} & S=-2,\\
					\text{Thm. \ref{thm_main} \eqref{thm_main_ii}:} & S=-\frac{\tau^2}{2k(\tau-k)},\ 1\leq k<\frac{\tau}{2},\\
					\text{Thm. \ref{thm_main} \eqref{thm_main_iii}:} & S=0\text{ for all values }a,b,c.
				\end{array}
			\end{equation*}
	\end{Prop}
	
	Note that the scalar curvature of the surfaces in Thm. \ref{thm_main} \eqref{thm_main_ii} is strictly smaller than $-2$ for all $1\leq k<\frac{\tau}{2}$, see Figure \ref{fig_curvatures_ii}.
		\begin{figure}[H]
			\centering
			\includegraphics[scale=0.3]{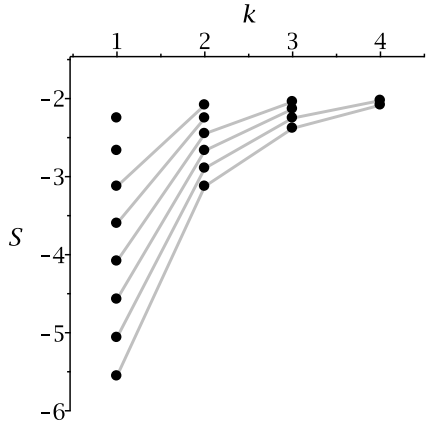}
			\caption{Scalar curvatures of the surfaces in Thm. \ref{thm_main} \eqref{thm_main_ii} for $3\leq\tau\leq 10$. The connected dots represent a fixed value $\tau$ with $1\leq k<\frac{\tau}{2}$.}\label{fig_curvatures_ii}
		\end{figure}
	\noindent
	For every fixed degree $\tau$ of the polynomials, the minimum of the scalar curvatures of special homogeneous surfaces is obtained by Thm. \ref{thm_main} \eqref{thm_main_ii} for $k=1$ with $S=-\frac{\tau^2}{2(\tau-1)}$.
	
\section{Preliminaries}
	In the following we will introduce the definitions and technicalities necessary to prove our results. We start out with the definition of the type of polynomials that we are studying.
	\begin{Def}
		Let $h:\mathbb{R}^{n+1}\to\mathbb{R}$ be a homogeneous polynomial. Then $h$ is called \textbf{hyperbolic} if there exists a point $p\in\{h>0\}$, such that the negative Hessian $-\partial^2h_p$ is of Minkowski type, that is has one negative and $n$ positive eigenvalues. Such a point $p$ is then called hyperbolic point of $h$, and we will denote the set of hyperbolic points of a hyperbolic polynomial $h$ with $\mathrm{hyp}(h)$. Two hyperbolic polynomials $h,\overline{h}$ are called \textbf{equivalent} if they are related by a linear transformation of the ambient space, that is if there exists $A\in\mathrm{GL}(n+1)$, such that $A^*h=\overline{h}$.
	\end{Def}
	
	The above definition implies that hyperbolic polynomials are of degree at least two. Also note that there is a canonical identification of real symmetric $\tau$-tensors on $\mathbb{R}^{n+1}$ and homogeneous polynomials on $\mathbb{R}^{n+1}$.
	
	\begin{rem}
		Hyperbolic polynomials in particular contain the set of strictly Lorentzian polynomials, which have recently gained much traction in pure mathematics \cite{BH}. This inclusion follows from \cite[Thm.\,2.16]{BH}. There is a more direct way however to see that this holds, which we will present now. Let $h$ be a strictly Lorentzian polynomial of degree $\tau\geq 2$ in $n$ real variables $x_1,\ldots,x_n$. By definition, $h$ has only positive coefficients, and for any fixed $1\leq j\leq n$, the bilinear form $-\partial^{n-2}_{x_j}h(\cdot,\cdot)$ is of Minkowski type. If $h(e_j)>0$, where $e_j$ denotes the $j$-th vector in the standard orthonormal basis of $\mathbb{R}^n$, it follows that $h$ is hyperbolic. If $h(e_j)=0$, we use that up to a positive scale $-\partial^{n-2}_{x_j}h(\cdot,\cdot)$ and $-\partial^2h_{e_j}$ coincide by Euler's theorem for homogeneous functions. Hence, $-\partial^2h_{e_j}$ having $n-1$-dimensional light cone and all coefficients of $h$ being positive implies that in any open neighbourhood of $e_j$ there exists some point $p$, such that $-\partial^2h_p$ is of Minkowski type and $h(p)>0$. Thus $p$ is a hyperbolic point of $h$.
		
		Further note that there is another common use for the term hyperbolic polynomial. In \cite{G,H} a complex homogeneous polynomial $h:\mathbb{C}^n\to\mathbb{C}$ is called hyperbolic with respect to a vector $v\in\mathbb{R}^n$, if $h(x+tv)$ has has $\mathrm{deg}(h)$ real zeros in $t$ for all $x\in\mathbb{R}^n\subset\mathbb{C}^n$. Our definition is consistent with previous works in our setting, and \cite{LSZH}.
	\end{rem}
	
	Next we will introduce the manifolds that we will be working with, which are defined to be contained in certain level sets of hyperbolic polynomials.
	
	\begin{Def}
		Let $h:\mathbb{R}^{n+1}\to\mathbb{R}$ be a hyperbolic polynomial of degree $\tau\geq 3$. A hypersurface $\mathcal{H}$ contained in the level set $\{h=1\}$ of a hyperbolic polynomial $h:\mathbb{R}^{n+1}\to\mathbb{R}$ is called \textbf{projective special real} (\textbf{PSR}) manifold if $\tau=3$, and \textbf{generalised projective special real} (\textbf{GPSR}) manifold if $\tau\geq 4$. When not restricting the degree to be either equal to or greater than $3$, we will write \textbf{(G)PSR} manifolds instead. Similarly, for their defining polynomials, we call two (G)PSR manifolds equivalent if they are related by a linear transformation of their ambient space.
	\end{Def}
	
	Note that two (G)PSR manifolds being equivalent automatically implies that their defining polynomials are equivalent via that same transformation. The other direction does in general not hold without further assumptions regarding the geometry of the involved (G)PSR manifolds. This can be seen by observing that any (G)PSR manifold is in no case equivalent to an open subset of itself that does not coincide with it.
	
	(G)PSR manifold carry a natural Riemannian metric given by the restriction of the negative Hessian of their respective defining polynomials to their tangent space. In symbols, this means that for any (G)PSR manifold $\mathcal{H}\subset\{h=1\}$, $g=-\partial^2h|_{T\mathcal{H}\times T\mathcal{H}}$ is Riemannian. This is a consequence of the hyperbolicity of $h$ and Euler's theorem for homogeneous functions. The latter implies that $-\partial^2 h_p(p,p)=\tau(\tau-1)$ for all $p\in\mathcal{H}$, where $\tau=\mathrm{deg}(h)$. Together with $-\partial^2h$ being of Minkowski type, and $T_p\mathcal{H}$ and $\mathbb{R}\cdot p$ being orthogonal with respect to $-\partial^2h$, this precisely means that $g>0$. The metric $g$ is referred to as the \textbf{centro-affine metric} \cite{CNS} of $\mathcal{H}$. As mentioned in the introduction, we will introduce a refinement of the term singular that incorporates information about connected components of $\mathrm{hyp}_1(h)$ for a hyperbolic polynomial $h$.
	
	\begin{Def}
		A (G)PSR manifold $\mathcal{H}\subset\{h=1\}\subset\mathbb{R}^{n+1}$ is called \textbf{singular at infinity} if there exists $p\in\left(\mathbb{R}_{>0}\cdot\mathcal{H}\right)\cap \{h=0\}$, such that $\D h_p=0$.
	\end{Def}
	
	As mentioned in the introduction, $\mathcal{H}\subset\{h=1\}$ being singular at infinity implies that the real projective algebraic variety $(h=0)$ is singular. Compared to the latter, our definition allows us to keep track of where the singularity is located, since the other direction need not necessarily hold even if we are assuming $\mathcal{H}$ to be a connected component of $\mathrm{hyp}_1(h)$. A concept related to the above definition is so-called \textbf{regular boundary behaviour} of (G)PSR manifolds. Following \cite{CNS}, a (G)PSR manifold $\mathcal{H}\subset\{h=1\}$ is said to have regular boundary behaviour of it is not singular at infinity and $-\partial^2 h$ is positive semi-definite on $T\left(\left(\mathbb{R}_{>0}\cdot\mathcal{H}\right)\setminus\{0\}\right)$ with only one-dimensional kernel.
	
	There are no general classification results for homogeneous real polynomials of degree at least three. Restricting to hyperbolic polynomials or, equivalently, (G)PSR manifolds is too lax of a restriction to obtain such results with currently available tools. However, by further restricting either the degree of the polynomials, the dimension of the ambient space, or geometrical aspects of the (G)PSR manifolds, partial results have been obtained. The following definition details the restriction to homogeneous spaces in our setting.
	
	\begin{Def}
		A (G)PSR manifold $\mathcal{H}\subset\{h=1\}$ of dimension $n$ is called \textbf{homogeneous} if there exists a Lie subgroup $G\subset\mathrm{GL}(n+1)$ that acts transitively on $\mathcal{H}$.
	\end{Def}
	
		The reason to restrict to subgroups of the linear transformations in the above definition is as follows. Linear transformations that leave the defining polynomial of a (G)PSR manifold $\mathcal{H}$ invariant preserve the centro-affine metric $g$ and consequently its Levi-Civita connection, and preserve the centro-affine connection $\nabla^\mathrm{ca}$, which is defined by the centro-affine Gau{\ss} equation
			\begin{equation*}
				\mathrm{D}_X Y = \nabla^\mathrm{ca}_X Y + g(X,Y)\xi
			\end{equation*}
		for all $X,Y\in\mathfrak{X}(\mathcal{H})$, where $\mathrm{D}$ denotes the flat connection of the ambient space and $\xi$ denotes the position vector field in the ambient space. For the other direction, we have the following lemma.
		
	\begin{Lem}
		Let $\mathcal{H}\subset\{h=1\}$ be a (G)PSR manifold and $F$ an isometry of $\mathcal{H}$. If $F$ preserves the centro-affine connection of $\mathcal{H}$, there exists a linear transformation $A$ of the ambient space, such that $A|_{\mathcal{H}}=F$.
		\begin{proof}
			If $F$ is an isometry of $\mathcal{H}\subset\mathbb{R}^{n+1}$ that additionally preserves $\nabla^{\mathrm{ca}}$, we obtain from the centro-affine Gau{\ss} equation that
				\begin{equation*}
					\mathrm{D}_{F_*X}(F_*Y)=F_*\mathrm{D}_X Y
				\end{equation*}
			for all $X,Y\in\mathfrak{X}(\mathcal{H})$. Let $\overline{F}$ denote the homogeneous extension of $F$ degree one to the cone $U=\mathbb{R}_{>0}\cdot\mathcal{H}$ spanned by $\mathcal{H}$, so that $\overline{F}(rp)=rF(p)$ for all $r>0$ and all $p\in\mathcal{H}$. Choose a frame $\{X_1,\ldots,X_n\}$ of $\mathfrak{X}(\mathcal{H}\cap V)$, where $V\subset U$ is a possibly smaller open cone, such that such a choice is possible. Then we define a frame $\{\xi,Y_1,\ldots,Y_n\}$ of $\mathfrak{X}(V)$, where $\xi$ again denotes the position vector field and $(Y_i)_{rp} = (rX_i)_p$ for all $r>0$, $p\in\mathcal{H}\cap V$, and all $1\leq i\leq n$. The vector fields $Y_1,\ldots,Y_n$ are tangential to the level sets of $h$, and we have $(\overline{F}_* Y_i)_{rp} = r (F_* X_i)_p$ for all $r>0$, $p\in\mathcal{H}\cap V$, and all $1\leq i\leq n$, and we obtain $\mathrm{D}_{\overline{F}_* Y_i}(\overline{F}_* Y_j) = \overline{F}_*\mathrm{D}_{Y_i} Y_j$ for all $1\leq i,j\leq n$. Observe that $(\overline{F}_*\xi)_q = \D \overline{F}_{\overline{F}^{-1}(q)}(\xi_{\overline{F}^{-1}(q)}) = \overline{F}(\overline{F}^{-1}(q))=\xi_q$ for all $q\in V$. Hence, we have for all $Z\in\mathfrak{X}(V)$
				\begin{equation*}
					\mathrm{D}_{\overline{F}_* Z}(\overline{F}_* \xi) = \mathrm{D}_{\overline{F}_* Z}\xi=\overline{F}_* Z=\overline{F}_* \mathrm{D}_Z \xi.
				\end{equation*}
			We further calculate using the torsion-freeness of the flat connection, $\overline{F}_*\xi = \xi$, and $[\overline{F}_* X,\overline{F}_* Y] = \overline{F}_*[X,Y]$ for all $X,Y\in\mathfrak{X}(V)$,
				\begin{equation*}
					\mathrm{D}_{\overline{F}_* \xi} (\overline{F}_* Y_i) = \mathrm{D}_{\overline{F}_* Y_i} \xi + [\overline{F}_*\xi,\overline{F}_* Y_i] = \overline{F}_* \left(Y_i + [\xi, Y_i]\right) = \overline{F}_* \left( \mathrm{D}_{Y_i} \xi + [\xi,Y_i]\right) = \overline{F}_* \mathrm{D}_\xi Y_i
				\end{equation*}
			for all $1\leq i\leq n$. This shows that $\overline{F}$ preserves the flat connection on $V$. From $V$ being open in the ambient space it follows that $\overline{F}$ is indeed a restriction of a linear map $A$, and that $A|_\mathcal{H}=F$ as claimed.
		\end{proof}
	\end{Lem}
	
	\begin{rem}
		Homogeneous PSR manifolds and their defining hyperbolic cubics have been completely classified for any dimension in \cite{DV}. The authors additionally study the homogeneous projective special and quaternionic K\"ahler manifolds that can be obtained via the supergravity r- and q-map, respectively, from homogeneous PSR manifolds. When restricting to curves, all hyperbolic polynomials $h:\mathbb{R}^2\to\mathbb{R}$ of degree $\tau\geq 3$ so that $\mathrm{hyp}_1(h)$ is a homogeneous space have been classified in \cite{L5}. Such curves are called special homogenous curves, and any such polynomial is equivalent to
			\begin{equation*}
				h=x^{\tau-k}y^k
			\end{equation*}
		for precisely one $k\in\left\{1,\ldots,\left\lfloor\frac{\tau}{2}\right\rfloor\right\}$. Furthermore, there are complete classifications of PSR curves \cite{CHM}, PSR surfaces \cite{CDL}, PSR manifolds with reducible defining polynomial \cite{CDJL}, and of GPSR curves with quartic defining polynomial \cite{L4}. When restricting to homogeneous (G)PSR manifolds, it has been shown in \cite{L2} that homogeneous PSR manifolds are singular at infinity, and the same has been shown for homogeneous GPSR manifolds with quartic defining polynomial in \cite{L4}. It is at this stage unknown if a similar statement holds for higher degrees of the hyperbolic defining polynomials, but we carefully expect a positive answer. At least for GPSR surfaces we will show that this in fact holds in Proposition \ref{prop_shs_sing_at_inf}.
	\end{rem}
	
	In this work, we will extend known classifications to homogeneous (G)PSR surfaces independent of the degree of the defining polynomials.
	
	\begin{Def}
		A homogeneous (G)PSR surface defined by a hyperbolic polynomial $h:\mathbb{R}^3\to\mathbb{R}$ of degree at least three is called \textbf{special homogeneous surface}.
	\end{Def}
	
	For hyperbolic cubics in three variables, the classification of PSR surfaces in \cite{CDL} technically already includes a classification of the two homogeneous cases, see \cite[Thm.\,1.1\,a)\,\&\,b)]{CDL}. Our way of proving Theorem \ref{thm_main} will reproduce these two cases.
	
	The next result is of technical nature. It allows us to study changes in the metric centro-affine metric $g$ near a given point, and furthermore an easy calculation of the curvature of the centro-affine metric $g$.
	
	\begin{Prop}\label{prop_standard_form}
		Let $h:\mathbb{R}^{n+1}\to\mathbb{R}$ be a hyperbolic polynomial of degree $\tau\geq 3$ and let $(x,y_1,\ldots,y_n)=(x,y^\mathrm{T})$ denote linear coordinates on $\mathbb{R}^{n+1}$. Then for every hyperbolic point $p$ of $h$, there exists a linear transformation of the ambient space $A\in\mathrm{GL}(n+1)$, such that $A\left(\begin{smallmatrix} 1\\0\end{smallmatrix}\right)=p$, where $\left(\begin{smallmatrix} 1\\0\end{smallmatrix}\right)=(1,0,\ldots,0)^\mathrm{T}$, and
			\begin{equation}\label{eqn_h_standard_form}
				A^*h=x^\tau -x^{\tau-2}\langle y,y\rangle + \sum\limits_{k=3}^\tau x^{\tau-k}P_k(y),
			\end{equation}
		where $\langle\cdot,\cdot\rangle$ denotes the induced standard Euclidean inner product on $\mathbb{R}^n$ and $P_k:\mathbb{R}^n\to\mathbb{R}$, $3\leq k\leq \tau$, are homogeneous polynomials of degree $k$, respectively.
		\begin{proof}
			\cite[Prop.\,3.1]{L2}
		\end{proof}
	\end{Prop}
	
	If a hyperbolic polynomial $h$ is of the form \eqref{eqn_h_standard_form}, we say that $h$ is in \textbf{standard form}. This name, however, is not supposed to imply that bringing $h$ to standard form is either unique in general, as in the terms $P_k(y)$ in general depend on the choice of the hyperbolic point $p$, or necessarily the cure-all for solving different types of problems. For classification problems in this work and e.g. in \cite{CDJL,L5}, it is, more often than not, not a good starting point. It has, however, been successfully used in the classification of hyperbolic quartics in two variables, cf. \cite{L4}. Moreover, when not restricting to homogeneous spaces, the standard form of hyperbolic polynomials will most likely be a good tool in understanding moduli space structures. First results in this direction for a subset of hyperbolic cubics, independent of the dimension, can be found in \cite{L3}. The advantage of having $h$ in the form \eqref{eqn_h_standard_form} is that the centro-affine metric $g$ at the point $\left(\begin{smallmatrix}1\\0\end{smallmatrix}\right)$ is up to scale simply the scalar product $\langle\cdot,\cdot\rangle$. Furthermore, the point $\left(\begin{smallmatrix}1\\0\end{smallmatrix}\right)$ locally minimises the Euclidean distance of $\{h=1\}$ and the origin. Additionally, the transformation $A$ can be chosen to be smooth, at least locally when ignoring possible monodromy issues that might arise globally on $\mathrm{hyp}_1(h)$, leading to an infinitesimal version of the terms $P_k(y)$ and consequently for partial derivatives of $g$ at $\left(\begin{smallmatrix}1\\0\end{smallmatrix}\right)$. For technical details see \cite[Sect.\,3]{L2}. Since the point $p\in\mathrm{hyp}(h)$ in Proposition \ref{prop_standard_form} is arbitrary, this allows us to obtain the following formula of the scalar curvature of (G)PSR manifolds. Note at this point that if $h$ is in standard form, the point $\left(\begin{smallmatrix}1\\0\end{smallmatrix}\right)$ is contained in $\mathrm{hyp}_1(h)$.
	
	\begin{Lem}\label{lem_scalar_curvature}
		Let $h:\mathbb{R}^{n+1}\to\mathbb{R}$ be a hyperbolic polynomial in standard form. Then the scalar curvature $S$ of the centro-affine metric on $\mathrm{hyp}_1(h)$ at $p=\left(\begin{smallmatrix}1\\0\end{smallmatrix}\right)$ is given by
			\begin{equation*}
				S(p)=n(1-n)+\frac{9\tau}{8}\sum\limits_{i,j,k=1}^{n}\left(-P_3(\partial_i,\partial_i,\partial_k)P_3(\partial_j,\partial_j,\partial_k)+P_3(\partial_i,\partial_j,\partial_k)^2\right),\label{eqn_GPSR_scal_formula}
			\end{equation*}
		where we have identified the cubic polynomial $P_3:\mathbb{R}^n\to\mathbb{R}$ with its symmetric trilinear form in $\mathrm{Sym}^3\left(\mathbb{R}^n\right)^*$, that is $P_3(\partial_i,\partial_j,\partial_k)=\frac{1}{6}\partial_i\partial_j\partial_k P_3(y)$ for all $1\leq i,j,k\leq n$, and denote by $\partial_k=\partial_{y_k}$ the $k$-th unit vector in $\mathbb{R}^n$ for all $1\leq k\leq n$.
		\begin{proof}
			\cite[Prop.\,3.9]{L2}.
		\end{proof}
	\end{Lem}
	
	Note that the scalar curvature at $\left(\begin{smallmatrix}1\\0\end{smallmatrix}\right)$ only depends on the term $P_3(y)$ in \eqref{eqn_h_standard_form}. For closed PSR surfaces, the two homogeneous examples minimise, respectively maximise, the scalar curvature in the following sense. For any closed PSR surface $\mathcal{H}$, the scalar curvature $S$ with respect to the centro-affine metric takes values in $\left[-\frac{9}{4},0\right]$. The extreme values are realised by the two homogeneous PSR surfaces corresponding to
		\begin{align*}
			h_1&=x^3-x\left(y^2+z^2\right)+\frac{2}{3\sqrt{3}}y^3+\frac{1}{\sqrt{3}}yz^2,\quad S=-\frac{9}{4},\\
			h_2&=x^3-x\left(y^2+z^2\right)+\frac{2}{3\sqrt{3}}y^3-\frac{2}{\sqrt{3}}yz^2,\quad S=0.
		\end{align*}
	The polynomial $h_1$ belongs to Thm. \ref{thm_main} \eqref{thm_main_ii}, and $h_2$ belongs to Thm. \ref{thm_main} \eqref{thm_main_iii}. For any other closed PSR surfaces, the scalar curvature has only values in $\left(-\frac{9}{4},0\right)$. For details and proofs of these statements see \cite[Prop.\,5.12]{L1}. It is thus a natural question to ask whether a similar kind of result holds for higher degree polynomials as well. The problem is that even for quartics, closed quartic GPSR surfaces are not classified yet, and the difficulty spike when comparing the classification of cubic curves \cite{CHM} with quartic curves \cite{L4} indicates that such a classification is a by no means easy task. We are however optimistic that it will eventually be obtained, and our present results will at least help with formulating conjectures about the respective scalar curvature behaviour. Also note that in \cite[Sect.\,4]{T} the curvature of a GPSR surface with quartic defining polynomial is studied. This is, to our knowledge, the only explicit study of a hyperbolic quartic in three real variables in this context.
	
	Lastly in this section, we will introduce two Lemmas which are very helpful in excluding polynomials during the proof of Theorem \ref{thm_main}.
	
	\begin{Lem}\label{lem_no_rays}
		Let $h:\mathbb{R}^{n+1}\to\mathbb{R}$ be a hyperbolic polynomial. Then the intersection of any line $L$ in the ambient space $\mathbb{R}^{n+1}$ with any connected component $\mathcal{H}$ of $\mathrm{hyp}_1(h)$ contains at most finitely many points.
			\begin{proof}
				When restricted to $L$, $(h-1)$ is a real polynomial in one variable and thus has either only finitely many or no zeros, or is identically zero along $L$. Suppose that the latter holds. Then we can write $L=\{p+tv\ |\ t\in\mathbb{R}\}$ for some $v\in\mathbb{R}^{n+1}\setminus\{0\}$ and some $p\in\mathrm{hyp}_1(h)$. Since $h$ is constant along $L$, $v$ is tangent to $\mathrm{hyp}_1(h)$ at $p$ since $\D h_p(v)=0$, and so $-\partial^2h_p(v,v)=0$. This is a contradiction to $-\partial^2h$ restricted to $T_p\;\! \mathrm{hyp}_1(h)$ being positive definite, and to thus the hyperbolicity of $h$.
			\end{proof}
	\end{Lem}
	
	\begin{Lem}\label{lem_H_bounded_away_from_0}
		Let $h:\mathbb{R}^{n+1}\to\mathbb{R}$ be a hyperbolic polynomial and let $\mathcal{H}$ be connected component of $\mathrm{hyp}_1(h)$. Then in any linear coordinates of the ambient space, $\mathcal{H}$ is bounded away from the origin with respect to the induced Euclidean norm.
			\begin{proof}
				By linear transformation being continuous, it suffices to prove this statement for a fixed choice of linear coordinates on $\mathbb{R}^{n+1}$. Suppose that $\mathcal{H}$ is not bounded away from the origin $0\in\mathbb{R}^{n+1}$. This implies that $0\in\overline{\mathcal{H}}$, which in turn by the continuity of $h$ would imply $h(0)=1$. This is a contradiction to $h$ being a homogeneous polynomial.
			\end{proof}
	\end{Lem}

\section{Proofs of our results}
	We will start with proving Theorem \ref{thm_main} and Proposition \ref{prop_scalar_curvature}. We split the proof in two cases, depending on the stabiliser of the polynomials being either one- or zero-dimensional. In the one-dimensional stabiliser case, our ansatz is to fix the generator $A$ of the stabiliser of the studied polynomials, which by the hyperbolicity condition restricts $A$ to be a generator of the circle group. We then classify all three-dimensional Lie subalgebras of $\mathfrak{gl}(3)$ that contain $A$ and permit the polynomials to be hyperbolic, and subsequently study the corresponding invariant polynomials. In the case of the stabiliser being discrete, we start with fixing the generators of the two-dimensional Lie algebra that our polynomials are supposed to be invariant under so that hyperbolicity cannot a priori be excluded, classify the different possibilities in terms of explicit generators, and then check that the invariant polynomials are hyperbolic. In any case, we then need to study which of the invariant polynomials $h$ in each considered case are pairwise inequivalent and describe the geometry of $\mathrm{hyp}_1(h)$. It is then a straight forward task to calculate the scalar curvature of the obtained special homogeneous surfaces. Thereafter we will prove Proposition \ref{prop_shs_sing_at_inf}.
	\\
	\vspace{0pt}
	\\
	\noindent
	\textit{Proof of Theorem \ref{thm_main} and Proposition \ref{prop_scalar_curvature}:}\\
	\vspace{-24pt}
	\subsection{One-dimensional stabiliser}\label{sect_1dimstab}
		Suppose that $\mathcal{H}\subset\{h=1\}$ is a special homogeneous surface with one-dimensional stabiliser. For any $p\in\mathcal{H}$, the stabiliser with respect to $p$ as a subgroup $H$ of $\mathrm{GL}(3)$ in particular preserves the bilinear form $-\partial^2 h_p|_{T_p\mathcal{H}\times T_p\mathcal{H}}$, which is positive definite by $p$ being a hyperbolic point of $\mathcal{H}$. Hence, the identity component $H_0$ of $H$ is a subgroup of $\mathrm{SO}(2)$, and for dimensional reasons coincides with the latter. In the following, we will assume without loss of generality that $(x,y,z)^\mathrm{T} = (0,0,1)^\mathrm{T}\in\mathcal{H}$ and that $H_0$ with respect to that point in generated by
			\begin{equation}\label{eqn_A_stab_1dim}
				A=\left(\begin{matrix}
					0 & -1 & 0\\
					1 & 0 & 0\\
					0 & 0 & 0
				\end{matrix}\right).
			\end{equation}
		We now write
			\begin{equation*}
				h=\sum\limits_{k=0}^\tau z^{\tau-k} P_k,
			\end{equation*}
		where $P_k$ is a homogeneous polynomial of degree $k$ in $x,y$ for all $0\leq k\leq \tau$. Then $h$ being invariant under $A$ is equivalent for all $P_k$ to be $\mathrm{SO}(2)$-invariant under the standard representation. The antisymmetry of homogeneous polynomials of odd degree hence implies that $P_k=0$ for all odd $k$. The real valued function $r=\sqrt{x^2+y^2}$ is homogeneous of degree $1$, $\mathrm{SO}(2)$-invariant, and for even $k$ we have that $r^k$ is a polynomial. Then $\frac{P_k}{r^k}$ being constant on $\mathbb{R}^2\setminus\{0\}$ implies that for $k$ even, $P_k$ and $r^k$ coincide up to scale. We can thus write
			\begin{equation}\label{eqn_polys_onedim_stabiliser}
				h=\sum\limits_{\ell=0}^{\left\lfloor\frac{\tau}{2}\right\rfloor} c_\ell z^{\tau-2\ell} (x^2+y^2)^{\ell},
			\end{equation}
		where $c_\ell\in\mathbb{R}$ for all $0\leq \ell\leq \left\lfloor\frac{\tau}{2}\right\rfloor$. We now need to find all such coefficients $c_\ell$, such that the connected component of $\{h=1\}$ containing $(x,y,z)^\mathrm{T}=(0,0,1)^\mathrm{T}$ is a homogeneous space.
		
		To do so, we will first determine all $3$-dimensional Lie subalgebras of $\mathfrak{gl}(3)$ that contain $A$ \eqref{eqn_A_stab_1dim}. For $B=(B_{ij})\in\mathfrak{gl}(3)\setminus\{0\}$ we have
			\begin{align}
				[A,B]&=\left(\begin{matrix}
					-B_{12}-B_{21} & B_{11}-B_{22} & -B_{23}\\
					B_{11}-B_{22} & B_{12}+B_{21} & B_{13}\\
					-B_{32} & B_{31} & 0
				\end{matrix}\right),\label{eqn_A_B_commutator_so1_stab}\\
				[A,[A,B]]&=\left(\begin{matrix}
					-2B_{11}+2B_{22} & -2B_{12}-2B_{21} & -B_{13}\\
					-2B_{12}-2B_{21} & 2B_{11}-2B_{22} & -B_{23}\\
					-B_{31} & -B_{32} & 0
				\end{matrix}\right).\notag
			\end{align}
		The first commutator $[A,B]$ is symmetric in the upper left $2\times2$ block. We can without loss of generality assume that this also holds for $B$, that is $B_{12}=B_{21}$, by the skew-symmetric property of $A$. Hence, the three matrices $A$, $B$, and $[A,B]$ are linearly independent if and only if $B$ and $[A,B]$ are linearly independent. Suppose that $B=r[A,B]$ for some $r\ne 0$. We quickly see that $B_{13}=B_{23}=B_{31}=B_{32}=0$, $B_{22}=-B_{11}$. But then $B_{11}=r(-2B_{21})=r^2(-2B_{11}+2B_{22})=-4r^2B_{11}$, implying $B_{11}=0$ and thereby $B=0$, which contradicts $B\ne 0$. Thus, $A$, $B$, and $[A,B]$ are linearly independent.
		
		For $\{A,B,[A,B]\}$ to generate a three-dimensional Lie subalgebra of $\mathfrak{gl}(3)$, it is a necessary condition that $[A,[A,B]]$ lies in their span. Since $[A,[A,B]]$ has a symmetric upper left $2\times2$ block, we need to check when $rB+s[A,B]=[A,[A,B]]$ has a solution in $r,s\in\mathbb{R}$. Suppose that $r=0$ and $s=0$ solves that equation. Then $B$ must be of the form
			\begin{equation}\label{eqn_B_vanishing_A_A_B}
				B=\left(\begin{matrix}
					B_{11} & -B_{21} & 0\\
					B_{21} & B_{11} & 0\\
					0 & 0 & B_{33}
				\end{matrix}\right)
			\end{equation}
		Since $B_{12}=B_{21}$, this implies $B_{21}=0$. By assumption, $(0,0,1)^\mathrm{T}$ is a hyperbolic point of $h$. Since $B\ne 0$ and $B$ is not an element of the Lie algebra of the stabiliser of that point by construction, $B_{33}\ne 0$ must hold. But then the image of
			\begin{equation*}
				t\mapsto e^{tB}\left(\begin{smallmatrix}0\\0\\1\end{smallmatrix}\right)
			\end{equation*}
		is a ray, which is contained in $\mathcal{H}$. This is a contradiction to Lemma \ref{lem_no_rays}. Thus, we can exclude the case $r=0$ and $s=0$, or equivalently $[A,[A,B]]=0$. Now suppose $[A,[A,B]]\ne 0$. Since both $[A,B]$ and $[A,[A,B]]$ have a symmetric upper left $2\times2$ block with alternating sign in the diagonal, we obtain that $[A,B]$ and $[A,[A,B]]$ must be linearly dependent, or $B_{22}=-B_{11}$. In the first case, that is $r=0$, we obtain by comparing the diagonal entries of $s[A,B]=[A,[A,B]]$ that
			\begin{equation*}
				sB_{21}=B_{11}-B_{22},\quad s(B_{11}-B_{22})=-4B_{21}.
			\end{equation*}
		Since $s\ne 0$ by $[A,[A,B]]\ne 0$, this implies that $s^2B_{21} = -4B_{21}$. This can only be fulfilled if $B_{21}=0$, which in turn implies $B_{22}=B_{11}$. Hence, the upper left $2\times 2$ blocks of $[A,B]$ and $[A,[A,B]]$ vanish and the remaining equations in $s[A,B]=[A,[A,B]]$ read
			\begin{equation*}
				-sB_{32} = -B_{31},\quad sB_{31} = -B_{32},\quad -sB_{23} = -B_{13},\quad sB_{13} = -B_{23}.
			\end{equation*}
		This implies that $s^2B_{32}=-B_{32}$ and $s^2 B_{23}=-B_{23}$. Since $s\ne 0$, $B_{23}=B_{32}=0$ and consequently also $B_{13}=B_{31}=0$. Hence, $B$ is again of the form \eqref{eqn_B_vanishing_A_A_B} and this case can also be excluded. We are thus left with the case $r\ne 0$. In that case, by the symmetry properties of the upper left $2\times2$ blocks of $[A,B]$ and $[A,[A,B]]$ we necessarily have $B_{22}=-B_{11}$ and $B_{33}=0$. Suppose that $s=0$. Then the upper left $2\times2$ block of $rB=[A,[A,B]]$ is fulfilled if either $B_{11}=0$ and $B_{12}=0$, or $r=-4$. In the latter case, the remaining equations of $rB=[A,[A,B]]$ imply that $B_{13}=B_{31}=B_{23}=B_{32}=0$. But then $B$ is an element of the Lie algebra of the stabiliser of $(0,0,1)^\mathrm{T}$ which we have excluded by assumption. Thus, for $s=0$, $B$ must be of the form
			\begin{equation}\label{eqn_B_case_r_ne_0_s_0}
				B=\left(\begin{matrix}
					0 & 0 & B_{13}\\
					0 & 0 & B_{23}\\
					B_{31} & B_{32} & 0
				\end{matrix}\right),
			\end{equation}
		and we see that $rB=[A,[A,B]]$ is fulfilled for $B\ne 0$ if and only if $r=-1$. Lastly, if both $r\ne 0$ and $s\ne 0$, again by using the symmetry argument from the case $r\ne 0$ and $s=0$ we obtain $B_{22}=-B_{11}$ and $B_{33}=0$. The diagonal entries of $rB+s[A,B]=[A,[A,B]]$ are fulfilled if and only if $(r+4)B_{11}=2sB_{21}$ and $(r+4)B_{21}=-2sB_{11}$, which implies $-(r+4)^2B_{21} = 2s^2B_{21}$. Hence, $B_{21}=0$, implying $B_{11}=0$. This means that $B$ is again of the form \eqref{eqn_B_case_r_ne_0_s_0}. The remaining equations of $rB+s[A,B]=[A,[A,B]]$ are fulfilled if and only if
			\begin{equation*}
				(r+1)B_{13}=s B_{23},\quad (r+1) B_{23}=-sB_{13},
			\end{equation*}
		and the above equation with rows and columns swapped. We use $r\ne 0$ and $s\ne 0$ to obtain $-(r+1)^2 B_{23}=s^2 B_{23}$ which shows that $B_{23}=0$ and consequently $B_{13}=0$. With the same argument, we also find $B_{32}=B_{31}=0$. But then $B=0$, a contradiction to our assumptions.
		
		Summarising, we have shown that $\{A,B,[A,B]\}$ might only generate a $3$-dimensional Lie subalgebra of $\mathfrak{gl}(3)$ with $A$ as in \eqref{eqn_A_stab_1dim} fulfilling the additional stabiliser condition, if $B$ is of the form \eqref{eqn_B_case_r_ne_0_s_0}. In that case, $B=-[A,[A,B]]$, and we further calculate
			\begin{equation*}
				[B,[A,B]]=\left(\begin{matrix}
					B_{23}B_{31} - B_{13}B_{32} & B_{13}B_{31} + B_{23}B_{32} & 0\\
					-B_{13}B_{31} - B_{23}B_{32} & B_{23}B_{31} - B_{13}B_{32} & 0\\
					0 & 0 & 0
				\end{matrix}\right)
			\end{equation*}
		Hence, $[B,[A,B]]$ and $A$ must be proportional to each other, meaning that $B_{23}B_{31} = B_{13}B_{32}$ must hold. To reduce this problem further, observe that we might act with the stabiliser of the point $(0,0,1)^\mathrm{T}$, which is given by matrices of the form
			\begin{equation*}
				\left(\begin{matrix}
					M & \\
					 & 1
				\end{matrix}\right)\in\mathrm{GL}(3),
			\end{equation*}
		where $M\in\mathrm{SO}(2)$, via the adjoint action on $B$. We can also rescale $B$. By doing so we can reduce our studies to the two cases
			\begin{equation*}
				\widetilde{B}=\left(\begin{matrix}
					0 & 0 & 1\\
					0 & 0 & 0\\
					0 & 0 & 0
				\end{matrix}\right),\quad\text{and}\quad
				\widehat{B}=\left(\begin{matrix}
					0 & 0 & B_{13}\\
					0 & 0 & B_{23}\\
					1 & 0 & 0
				\end{matrix}\right).
			\end{equation*}
		We can immediately exclude the case $\widetilde{B}$ by noting that every non-constant orbit of the action induced by $\widetilde{B}$ is a straight line and using Lemma \ref{lem_no_rays}. We are thus left with the case $B=\widehat{B}$ and need to determine all polynomials $h$ of the form \eqref{eqn_polys_onedim_stabiliser} that are invariant under such a $B$ for some choice of $B_{13},B_{23}\in\mathbb{R}$.
		
		We have by the infinitesimal symmetry property of $B$
			\begin{equation*}
				\D h_{\left(\begin{smallmatrix}x\\y\\z\end{smallmatrix}\right)}\left(B\cdot \left(\begin{smallmatrix}x\\y\\z\end{smallmatrix}\right)\right) =z(B_{13}\partial_x h + B_{23}\partial_y h) + x\partial_z h=0.
			\end{equation*}
		Since the only terms in $z(B_{13}\partial_x h + B_{23}\partial_y h) + x\partial_z h$ that contain summands with odd degrees in $y$ occur in $zB_{23}\partial_y h$, every summand in the latter is in fact of odd degree in $y$, and since $\partial_y h$ is not identically zero by the hyperbolicity of $h$, we conclude that $B_{23}=0$. By rescaling in the $z$-variable and using an overall scale for $B$, we thus might further assume that $B_{13}=\pm1$. If $B_{13}=-1$, the set $\{A,B,[A,B]\}$ generates $\mathfrak{so}(3)$, which has only compact orbits and thus cannot act transitively on a special homogeneous surface. Hence, $B_{13}=1$ so that
			\begin{equation*}
				B=\left(\begin{matrix}
					0 & 0 & 1\\
					0 & 0 & 0\\
					1 & 0 & 0
				\end{matrix}\right).
			\end{equation*}
		
		We will now show that any hyperbolic polynomial $h$ that is invariant under $B$ has even degree $\tau$. To do so, observe that the induced action of $B$ on $h$ and restricting to the plane $\{y=0\}$ commute. This and $B$ acting non-trivially on $\{y=0\}$ means that the connected component of $\{h=1,\, y=0\}$ that contains the point $(0,0,1)^\mathrm{T}$ is a special homogeneous curve, allowing us to use their already available classification \cite[Thm.\,1.1]{L5}. Any connected special homogeneous curve contained in a hyperbolic polynomial of degree $\tau$ is equivalent to the connected component $\overline{\mathcal{H}}$ of $\{\widehat{h}:=x^{\tau-k}y^k=1\}$ that contains the point $(x,y)^\mathrm{T}=(1,1)^\mathrm{T}$ for precisely one $k\in\left\{1,\ldots,\left\lfloor\frac{\tau}{2}\right\rfloor\right\}$. Note that $\overline{\mathcal{H}}\subset\{x>0,\, y>0\}$. The identity component of the symmetry group of $\overline{h}$ is generated by
			\begin{equation*}
				\overline{B}:=\left(\begin{matrix}
					1 & 0\\
					0 & \frac{k-\tau}{k}
				\end{matrix}\right).
			\end{equation*}
		Hence, up to a coordinate change, $\overline{B}$ is precisely
			\begin{equation*}
				B|_{\{y=0\}}=\left(\begin{matrix}
					0 & 1\\
					1 & 0
				\end{matrix}\right),
			\end{equation*}
		and this implies that $1+\frac{k-\tau}{k}=\mathrm{tr}(\overline{B})=\mathrm{tr}\left(B_{\{y=0\}}\right)=0$. This can only be fulfilled if $\tau$ is even and $k=\frac{\tau}{2}$, showing that $\tau$ being even is a necessary requirement for the existence of a hyperbolic polynomial invariant under $A$, $B$, and $[A,B]$. We now make the following ansatz. We claim that for $\tau$ even,
			\begin{equation}\label{eqn_h_final_so1_stab}
				h=\left(z^2-x^2-y^2\right)^{\frac{\tau}{2}}
			\end{equation}
		is a hyperbolic polynomial with the desired symmetries. We see that $h$ is indeed of the form \eqref{eqn_polys_onedim_stabiliser} and we quickly check that $(0,0,1)^\mathrm{T}$ is a hyperbolic point of $h$. Furthermore, the term $z^2-x^2-y^2$ and consequently $h$ are invariant under $\mathrm{SO}(2,1)^{+}$ with generators $A$, $B$, and $[A,B]$. Hence, the connected special homogeneous surface $\mathcal{H}$ is diffeomorphic to $\mathrm{SO}(2,1)^{+}/\mathrm{SO}(2)$. Note that up to switching variables, $h$ in \eqref{eqn_h_final_so1_stab} coincides with the polynomial in Thm. \ref{thm_main} \eqref{thm_main_i}. In order to see that $h$ in \eqref{eqn_h_final_so1_stab} is the unique invariant under $A$, $B$, and $[A,B]$, suppose that there is another such polynomial $\overline{h}$. Then, by construction, $(0,0,1)^\mathrm{T}$ must be a hyperbolic point of $\overline{h}$, implying that $h$ and $\overline{h}$ coincide on the connected component $\overline{H}$ of $\mathrm{hyp}_1(h)$ that contains $(0,0,1)^\mathrm{T}$. By homogeneity, $h$ and $\overline{h}$ must thus coincide on the open set $\mathbb{R}_{>0}\cdot \mathcal{H}$, and from them being polynomials it follows they must coincide on the whole ambient space. Hence, $\overline{h}=h$.
		
		It remains to determine the full symmetry group $G^h$ of $\mathrm{hyp}_1(h)$ with $h$ as in \eqref{eqn_h_final_so1_stab} and to calculate the scalar curvature of $\mathcal{H}$. For $\frac{\tau}{2}$ odd, $G^h$ and the symmetry group of $z^2-x^2-y^2$, that is $\mathrm{O}(2,1)$, coincide. We need to check that we do not get any additional symmetries for $\frac{\tau}{2}$ even. In that case $\{h>0\}$ has three connected components, two of which contain connected special homogeneous surfaces. The third one is given by $\{z^2-x^2-y^2<0\}$. But the negative Hessian of $h$ at any point in that set has signature $(1,2)$, that is one positive and two negative eigenvalues. Hence, $\{z^2-x^2-y^2<0\}$ does not contain a special homogeneous surface. We conclude that for all $\tau\geq 4$, $\tau$ even, $G^h=\mathrm{O}(2,1)$ and $\mathrm{hyp}_1(h)$ has two connected components which are isometric with respect to the restriction of $-\partial^2h$.
		
		In order to calculate the scalar curvature of $\mathrm{hyp}_1(h)$, we will transform $h$ in \eqref{eqn_h_final_so1_stab} to standard form \eqref{eqn_h_standard_form}. But, up to rescaling in $x,y$ and switching variables, $h$ is already of standard form, that is
			\begin{equation*}
				h= z^\tau-\frac{\tau}{2}z^{\tau-2}\left(x^2+y^2\right)+\frac{\tau(\tau-2)}{8}z^{\tau-4}\left(x^2+y^2\right)^2+\text{terms of lower order in }z.
			\end{equation*}
		From the above formula it follows for every $\tau\geq 4$ that the term $P_3$ in \eqref{eqn_h_standard_form} vanishes. Thus, by Lemma \ref{lem_scalar_curvature} and the homogeneity of special homogeneous surfaces we obtain that the scalar curvature $S$ at each point in $\mathrm{hyp}_1(h)$ with respect to the centro-affine metric is given by $S=-2$.
		
	\subsection{Locally free symmetry}
		The action of the automorphism group $G^h$ on a special homogeneous surface $\mathcal{H}\subset\{h=1\}$ is locally free if its stabiliser $\mathrm{stab}^h$ satisfied $\dim(\mathrm{stab}^h)=0$. In that case, $\mathcal{H}$ is isomorphic to a two-dimensional Lie group $G\subset\mathrm{GL}(3)$. The Lie algebra $\mathfrak{g}$ of $G$ is either affine or abelian, and we will study both cases separately. In both cases, we will start with assuming one generator to be of a certain form. Suppose that $\mathfrak{g}$ is generated by $A,B\in\mathfrak{gl}(3)$. We can, without loss of generality after rescaling and acting adjointly with $\mathrm{GL}(3)$, assume that either $A$ has only real eigenvalues and is given by one of
			\begin{align*}
				&A_1=\left(\begin{matrix}
					\lambda & 1 & 0\\
					0 & \lambda & 1\\
					0 & 0 & \lambda
				\end{matrix}\right),\quad\lambda\in\{0,1\},\\
				&A_2=\left(\begin{matrix}
					\lambda & 1 & 0\\
					0 & \lambda & 0\\
					0 & 0 & \mu
				\end{matrix}\right),\quad\lambda\in\{0,1\},\ \mu\in\mathbb{R},\\
				&A_3=\left(\begin{matrix}
					1 & 0 & 0\\
					0 & \lambda & 0\\
					0 & 0 & \mu
				\end{matrix}\right),\quad\lambda,\mu\in\mathbb{R},
			\end{align*}
		or that $A$ has one complex eigenvalue with non-vanishing imaginary part and is given by
			\begin{equation*}
				A_4=\left(\begin{matrix}
					\lambda & -1 & 0\\
					1 & \lambda & 0\\
					0 & 0 & \mu
				\end{matrix}\right),\quad \lambda,\mu\in\mathbb{R}.
			\end{equation*}
		We will now show that some of the above generators can be excluded by using Lemmas \ref{lem_no_rays} and \ref{lem_H_bounded_away_from_0}. To do so, we need to study $e^{tA_i}$ for $1\leq i\leq 4$, $t\in\mathbb{R}$.
		
		For $\lambda=0$, $A_1$ can not be excluded and it does in fact appear in our later classification, cf. \eqref{eqn_a_1_to_5}. But $A_1$ with $\lambda=1$ can be excluded. To see this, we calculate for $A_1$, $\lambda=1$,
			\begin{equation*}
				e^{tA_1}=e^t\left(\begin{matrix}
					1 & t & \frac{t^2}{2}\\
					0 & 1 & t\\
					0 & 0 & 1
				\end{matrix}\right).
			\end{equation*}
		Hence, for $\lambda=1$ and all $p\in\mathbb{R}^3$, we have $\lim\limits_{t\to-\infty} e^{tA_1}p=0$. This in particular holds for all $p\in\mathcal{H}$ and would thereby imply that the origin is in the closure of $\mathcal{H}$. This is a contradiction to Lemma \ref{lem_H_bounded_away_from_0} and we can thus exclude $A_1$ with $\lambda=1$.
		
		For $A_2$, suppose first that $\lambda=0$. Then
			\begin{equation*}
				e^{tA_2}=\left(\begin{matrix}
					1 & t & 0\\
					0 & 1 & 0\\
					0 & 0 & e^{\mu t}
				\end{matrix}\right).
			\end{equation*}
		The existence of a point in the ambient space $\mathbb{R}^3$ that has an orbit that is not a line, which is a necessary condition for hyperbolicity of said point, implies that $\mu\ne0$. We can thus exclude $A_2$ for $\lambda=0$, $\mu=0$. For $\lambda=1$ we have
			\begin{equation*}
				e^{tA_2}=\left(\begin{matrix}
					e^t & te^t & 0\\
					0 & e^t & 0\\
					0 & 0 & e^{\mu t}
				\end{matrix}\right).
			\end{equation*}
		Lemma \ref{lem_H_bounded_away_from_0} implies that $\mu<0$ is a necessary condition for $\mathcal{H}$ to be a connected component of a hyperbolic level set, and we can thus exclude $A_2$ for $\lambda=1$ and $\mu\geq0$.
		
		The generator $A_3$ has as only non-trivial orbits lines for $\lambda=\mu=0$, meaning that we can exclude it by Lemma \ref{lem_no_rays}. We can also exclude $\lambda>0$, $\mu>0$ since
			\begin{equation*}
				e^{tA_3}=\left(\begin{matrix}
					e^t & 0 & 0\\
					0 & e^{\lambda t} & 0\\
					0 & 0 & e^{\mu t}
				\end{matrix}\right),
			\end{equation*}
		which in this case would imply that $0$ is in the closure of $\mathcal{H}$ and, hence, contradicts Lemma \ref{lem_H_bounded_away_from_0}. After switching variables in the ambient space, the only cases that cannot be excluded at this stage are thus $\lambda\in\mathbb{R}$, $\mu<0$. After possibly switching the first and third variable of the ambient space and rescaling, we might assume $\lambda\geq 0$ instead of $\lambda\in\mathbb{R}$.
		
		Lastly, for $A_4$ and $\lambda=0$, neither value for $\mu\in\mathbb{R}$ can be excluded by Lemma \ref{lem_no_rays} or Lemma \ref{lem_H_bounded_away_from_0}. After possibly switching the first and second variable in the ambient space and after an overall sign change, we can for $\lambda=0$ without loss of generality assume that $\mu\leq0$.	Similarly, we can for $\lambda\ne 0$ without loss of generality assume that $\lambda>0$. In that case,
			\begin{equation*}
				e^{tA_4}=\left(\begin{matrix}
					e^t\cos(t) & -e^t\sin(t) & 0\\
					e^t\sin(t) & e^t\cos(t) & 0\\
					0 & 0 & e^{\mu t}
				\end{matrix}\right).
			\end{equation*}
		Hence, $\mu\leq0$ is a necessary condition since for $\mu>0$ we have $\lim\limits_{t\to-\infty}e^{tA_4}p=0$ for all $p\in\mathbb{R}^3$ which would contradict Lemma \ref{lem_H_bounded_away_from_0}. However, for $\mu=0$ and $\lambda\ne 0$, the orbit of any point $p$ not contained in $\mathbb{R}\cdot(0,0,1)^\mathrm{T}$, in particular every hyperbolic point of $h$, would be an unbounded subset of a plane $E$ in $\mathbb{R}^3$. Then $h$, restricted to any line in $E$, would have infinitely many pre-images of $h(p)$, which implies that $h$ needs to be constant on $E$. This contradicts the existence of hyperbolic points of $h$ in $E$. Since the set of hyperbolic points of $h$ is open, this would contradict the hyperbolicity of $h$. Hence, for $\lambda \ne 0$ we can exclude $\mu=0$.
		
		Summarising, we have the following possibilities for one generator of the Lie algebra $\mathfrak{g}$:
			\begin{align}
				&a_1=\left(\begin{matrix}
					0 & 1 & 0\\
					0 & 0 & 1\\
					0 & 0 & 0
				\end{matrix}\right),\quad
				a_2=\left(\begin{matrix}
					0 & 1 & 0\\
					0 & 0 & 0\\
					0 & 0 & \mu
				\end{matrix}\right),\mu\in\mathbb{R}\setminus\{0\},\quad
				a_3=\left(\begin{matrix}
					1 & 1 & 0\\
					0 & 1 & 0\\
					0 & 0 & \mu
				\end{matrix}\right),\ \mu<0,\notag\\
				&a_4=\left(\begin{matrix}
					1 & 0 & 0\\
					0 & \lambda & 0\\
					0 & 0 & \mu
				\end{matrix}\right),\ \lambda\geq0,\mu<0,\quad
				a_5=\left(\begin{matrix}
					\lambda & -1 & 0\\
					1 & \lambda & 0\\
					0 & 0 & \mu
				\end{matrix}\right)\ \lambda\geq0,\mu<0,\notag\\
				&a_6=\left(\begin{matrix}
					0 & -1 & 0\\
					1 & 0 & 0\\
					0 & 0 & 0
				\end{matrix}\right).\label{eqn_a_1_to_5}
			\end{align}
		
		\subsubsection{$\mathfrak{g}$ non-abelian}
			Suppose that $\mathfrak{g}$ is non-abelian. Then, for dimensional reasons, $\mathfrak{g}$ must be the affine two-dimensional Lie algebra. Choose a basis $A,B$ of $\mathfrak{g}$ with $[A,B]=A$. Then $A$ is linearly equivalent to some $a_i$, $1\leq i\leq 6$. In the first step we need to find all $B\in\mathfrak{g}$ linearly independent from a given $A=a_i$, $1\leq i\leq 6$, in \eqref{eqn_a_1_to_5}, so that $[a_i,B]=a_i$. We can immediately exclude all $a_i$ for $2\leq i\leq 5$. This follows from the fact that in these cases $(a_i)_{33}\ne 0$, but for all $B=(B_{ij})$
				\begin{equation*}
					{([a_i,B])}_{33}=0.
				\end{equation*}
			Hence, we only need to consider $a_1$ and $a_6$. We start with $a_1$ and find that $[a_1,B]=a_1$ holds if and only if $B$ is of the form
				\begin{equation*}
					B=\left(\begin{matrix}
						B_{11} & B_{12} & B_{13}\\
						0 & B_{11}+1 & B_{12}\\
						0 & 0 & B_{11}+2
					\end{matrix}\right).
				\end{equation*}
			We can without loss of generality assume that $B_{12}=0$ by replacing $B\to B-B_{12}a_1$. Then we can write $B$ as
				\begin{equation}\label{eqn_B_affine}
					B=\left(\begin{matrix}
						b-1 & 0 & c\\
						0 & b & 0\\
						0 & 0 & b+1
					\end{matrix}\right),\quad b,c\in\mathbb{R}.
				\end{equation}
			and we see that $B$ has three distinct real eigenvalues $b-1,b,b+1$, independent of $b,c\in\mathbb{R}$. Thus, Lemma \ref{lem_H_bounded_away_from_0} restricts $b$ so that at least one of these eigenvalues is positive and one is negative and, hence, $b\in(-1,1)$ is a necessary condition. In the next step we will determine all homogeneous polynomials $h$ of degree at least three that are invariant under $a_1$ and $B$, and check which ones are hyperbolic.
			
			A homogeneous polynomial
				\begin{equation}\label{eqn_h_form_affine_starting_point}
					h=\sum\limits_{k=0}^\tau\sum\limits_{\ell=0}^k P_{k \ell} z^{\tau-k}x^{k-\ell}y^\ell,
				\end{equation}
			$P_{k \ell}\in\mathbb{R}$ for all $k\in\{0,\ldots,\tau\}$, $\ell\in\{0,\ldots,k\}$, is invariant under $a_1$ if and only if
				\begin{align*}
					&\sum\limits_{k=1}^\tau\sum\limits_{\ell=0}^k (k-\ell)P_{k \ell}z^{\tau-k}x^{k-\ell-1}y^{\ell+1}\\
					&+ \sum\limits_{k=1}^\tau\sum\limits_{\ell=1}^k\ell P_{k \ell} z^{\tau-k+1}x^{k-\ell}y^{\ell-1}=0.
				\end{align*}
			By shifting indices and combining sums, we find that the above equation is equivalent to
				\begin{align}
					&\sum\limits_{k=1}^{\tau-1}\sum\limits_{\ell=1}^k\bigg(\underbrace{(k-\ell+1)}_{\ne 0} P_{k (\ell-1)} + \underbrace{(\ell+1)}_{\ne 0} P_{(k+1) (\ell+1)}\bigg) z^{\tau-k}x^{k-\ell}y^\ell\label{eqn_P_recursion}\\
					&+\sum\limits_{\ell=1}^\tau \underbrace{(\tau-\ell+1)}_{\ne 0} P_{\tau (\ell-1)} x^{\tau-\ell}y^{\ell} + \sum\limits_{k=1}^{\tau-1} P_{(k+1) 1}z^{\tau-k}x^k + P_{1 1} z^\tau =0.\notag
				\end{align}
			Hence, $P_{k 1}=0$ for all $k\in\{1,\ldots,\tau\}$ and $P_{\tau\ell}=0$ for all $\ell\in\{0,\tau-1\}$, see Figures \ref{fig_kl_1_odd}, \ref{fig_kl_1_even} for a visualisation of which indices must vanish for an odd and an even example of $\tau$.\\
			\begin{minipage}{0.45\textwidth}
				\begin{figure}[H]
					\begin{center}
						\begin{tikzpicture}
							\begin{axis}[
								axis x line=center,
								axis y line=center,
								xtick={0,1,...,7},
								ytick={0,1,...,7},
								xlabel={$k$},
								ylabel={$\ell$},
								xlabel style={below right},
								ylabel style={above left},
								xmin=-0.5,
								xmax=7.5,
								ymin=-0.5,
								ymax=7.5,
								scatter/classes={
									a={mark=o,draw=black,mark size=3pt},b={mark=*,draw=black,mark size=3pt}
								}
							]
							\addplot[scatter,only marks,scatter src=explicit symbolic]
							table[meta=label]{
								x y label
								0 0 a
								1 0 a
								2 0 a
								3 0 a
								4 0 a
								5 0 a
								6 0 a
								7 0 b
								1 1 b
								2 1 b
								3 1 b
								4 1 b
								5 1 b
								6 1 b
								7 1 b
								2 2 a
								3 2 a
								4 2 a
								5 2 a
								6 2 a
								7 2 b
								3 3 a
								4 3 a
								5 3 a
								6 3 a
								7 3 b
								4 4 a
								5 4 a
								6 4 a
								7 4 b
								5 5 a
								6 5 a
								7 5 b
								6 6 a
								7 6 b
								7 7 b
								    };
							\end{axis}
						\end{tikzpicture}
						\caption{
							Table of the indices $(k,\ell)$ for $\tau=7$ in \eqref{eqn_h_form_affine_starting_point} that are marked with a solid black dot represent $P_{k \ell}=0$.
						}
						\label{fig_kl_1_odd}
					\end{center}
				\end{figure}
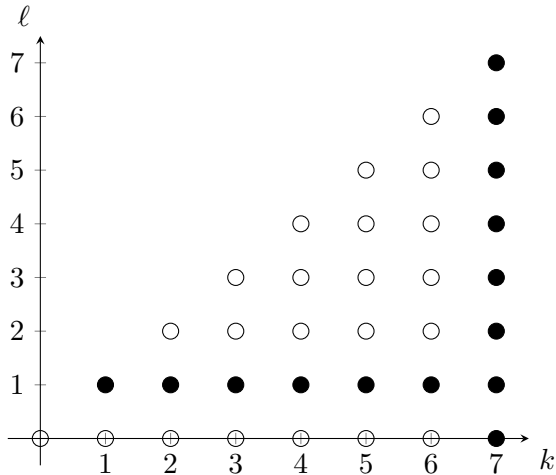
			\end{minipage}
			\hfill
			\begin{minipage}{0.45\textwidth}
				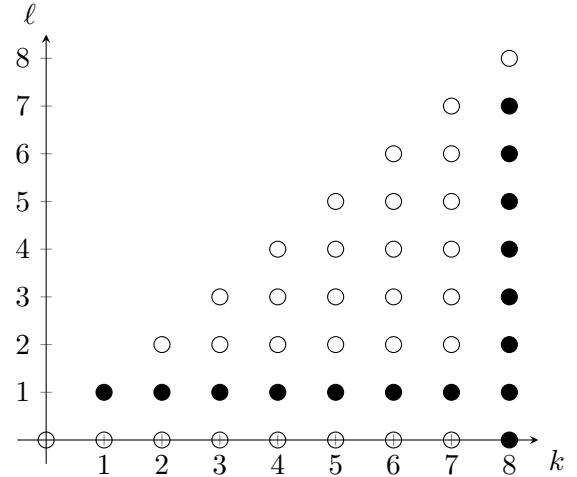
\begin{figure}[H]
					\begin{center}
						\begin{tikzpicture}
							\begin{axis}[
								axis x line=center,
								axis y line=center,
								xtick={0,1,...,8},
								ytick={0,1,...,8},
								xlabel={$k$},
								ylabel={$\ell$},
								xlabel style={below right},
								ylabel style={above left},
								xmin=-0.5,
								xmax=8.5,
								ymin=-0.5,
								ymax=8.5,
								scatter/classes={
									a={mark=o,draw=black,mark size=3pt},b={mark=*,draw=black,mark size=3pt}
								}
							]
							\addplot[scatter,only marks,scatter src=explicit symbolic]
							table[meta=label]{
								x y label
								0 0 a
								1 0 a
								2 0 a
								3 0 a
								4 0 a
								5 0 a
								6 0 a
								7 0 a
								8 0 b
								1 1 b
								2 1 b
								3 1 b
								4 1 b
								5 1 b
								6 1 b
								7 1 b
								8 1 b
								2 2 a
								3 2 a
								4 2 a
								5 2 a
								6 2 a
								7 2 a
								8 2 b
								3 3 a
								4 3 a
								5 3 a
								6 3 a
								7 3 a
								8 3 b
								4 4 a
								5 4 a
								6 4 a
								7 4 a
								8 4 b
								5 5 a
								6 5 a
								7 5 a
								8 5 b
								6 6 a
								7 6 a
								8 6 b
								7 7 a
								8 7 b
								8 8 a
									};
							\end{axis}
						\end{tikzpicture}
						\caption{
							Compared with Figure \ref{fig_kl_1_odd}, note that $P_{\tau\tau}$ need not necessarily vanish at this point for $\tau$ even, in this example $\tau=8$.
						}
						\label{fig_kl_1_even}
					\end{center}
				\end{figure}
			\end{minipage}
			\\
			\noindent
			From
				\begin{equation}\label{eqn_k_ell_implications}
					\underbrace{(k-\ell+1)}_{\ne 0} P_{k (\ell-1)} + \underbrace{(\ell+1)}_{\ne 0} P_{(k+1) (\ell+1)}=0
				\end{equation}
			for all $k\in\{1,\ldots,\tau-1\}$, $\ell\in\{1,\ldots,k\}$ we inductively obtain $P_{k \ell}=0$ for all odd $\ell$ with $\ell\leq\tau$, and all $k\in\{\ell,\ldots,\tau\}$, see Figures \ref{fig_kl_2_odd}, \ref{fig_kl_2_even}.\\
			\begin{minipage}{0.45\textwidth}
				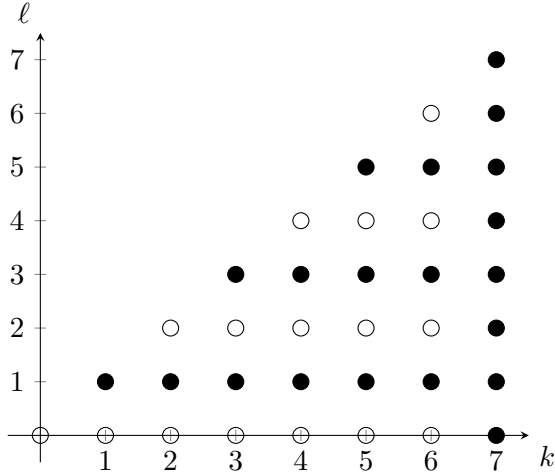
\begin{figure}[H]
					\begin{center}
						\begin{tikzpicture}
							\begin{axis}[
								axis x line=center,
								axis y line=center,
								xtick={0,1,...,7},
								ytick={0,1,...,7},
								xlabel={$k$},
								ylabel={$\ell$},
								xlabel style={below right},
								ylabel style={above left},
								xmin=-0.5,
								xmax=7.5,
								ymin=-0.5,
								ymax=7.5,
								scatter/classes={
									a={mark=o,draw=black,mark size=3pt},b={mark=*,draw=black,mark size=3pt}
								}
							]
							\addplot[scatter,only marks,scatter src=explicit symbolic]
							table[meta=label]{
								x y label
								0 0 a
								1 0 a
								2 0 a
								3 0 a
								4 0 a
								5 0 a
								6 0 a
								7 0 b
								1 1 b
								2 1 b
								3 1 b
								4 1 b
								5 1 b
								6 1 b
								7 1 b
								2 2 a
								3 2 a
								4 2 a
								5 2 a
								6 2 a
								7 2 b
								3 3 b
								4 3 b
								5 3 b
								6 3 b
								7 3 b
								4 4 a
								5 4 a
								6 4 a
								7 4 b
								5 5 b
								6 5 b
								7 5 b
								6 6 a
								7 6 b
								7 7 b
								    };
							\end{axis}
						\end{tikzpicture}
						\caption{
							From \eqref{eqn_k_ell_implications} we obtain that more values $P_{k\ell}$ must vanish in comparison with Figure \ref{fig_kl_1_odd}.
						}
						\label{fig_kl_2_odd}
					\end{center}
				\end{figure}
			\end{minipage}
			\hfill
			\begin{minipage}{0.45\textwidth}
				\begin{figure}[H]
					\begin{center}
						\begin{tikzpicture}
							\begin{axis}[
								axis x line=center,
								axis y line=center,
								xtick={0,1,...,8},
								ytick={0,1,...,8},
								xlabel={$k$},
								ylabel={$\ell$},
								xlabel style={below right},
								ylabel style={above left},
								xmin=-0.5,
								xmax=8.5,
								ymin=-0.5,
								ymax=8.5,
								scatter/classes={
									a={mark=o,draw=black,mark size=3pt},b={mark=*,draw=black,mark size=3pt}
								}
							]
							\addplot[scatter,only marks,scatter src=explicit symbolic]
							table[meta=label]{
								x y label
								0 0 a
								1 0 a
								2 0 a
								3 0 a
								4 0 a
								5 0 a
								6 0 a
								7 0 a
								8 0 b
								1 1 b
								2 1 b
								3 1 b
								4 1 b
								5 1 b
								6 1 b
								7 1 b
								8 1 b
								2 2 a
								3 2 a
								4 2 a
								5 2 a
								6 2 a
								7 2 a
								8 2 b
								3 3 b
								4 3 b
								5 3 b
								6 3 b
								7 3 b
								8 3 b
								4 4 a
								5 4 a
								6 4 a
								7 4 a
								8 4 b
								5 5 b
								6 5 b
								7 5 b
								8 5 b
								6 6 a
								7 6 a
								8 6 b
								7 7 b
								8 7 b
								8 8 a
   									};
							\end{axis}
						\end{tikzpicture}
						\caption{
							As for $\tau$ odd in Figure \ref{fig_kl_2_odd},  $P_{k\ell}$ vanishes for more indices $(k,\ell)$ for $\tau$ even.
						}
						\label{fig_kl_2_even}
					\end{center}
				\end{figure}
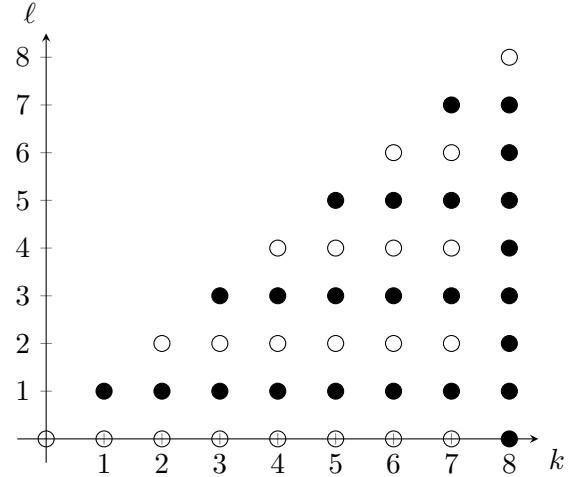
			\end{minipage}
			\\
			\noindent
			Similarly, we find that $P_{k\ell}=0$ for $\ell$ even, $\ell<\tau$, and all $k\in\left\{\tau-\lfloor\frac{\tau-\ell-1}{2}\rfloor,\ldots,\tau\right\}$, see Figures \ref{fig_kl_3_odd}, \ref{fig_kl_3_even}.\\
			\begin{minipage}{0.45\textwidth}
				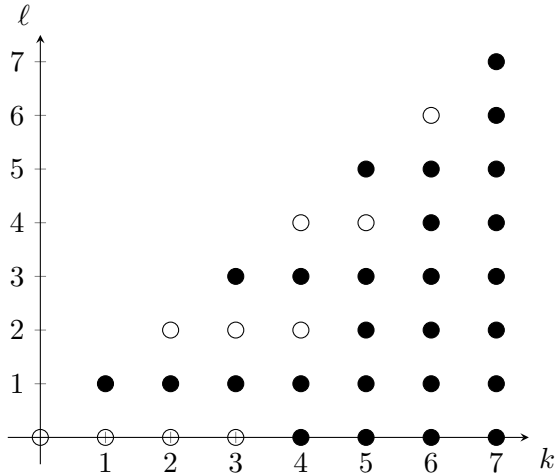
\begin{figure}[H]
					\begin{center}
						\begin{tikzpicture}
							\begin{axis}[
								axis x line=center,
								axis y line=center,
								xtick={0,1,...,7},
								ytick={0,1,...,7},
								xlabel={$k$},
								ylabel={$\ell$},
								xlabel style={below right},
								ylabel style={above left},
								xmin=-0.5,
								xmax=7.5,
								ymin=-0.5,
								ymax=7.5,
								scatter/classes={
									a={mark=o,draw=black,mark size=3pt},b={mark=*,draw=black,mark size=3pt}
								}
							]
							\addplot[scatter,only marks,scatter src=explicit symbolic]
							table[meta=label]{
								x y label
								0 0 a
								1 0 a
								2 0 a
								3 0 a
								4 0 b
								5 0 b
								6 0 b
								7 0 b
								1 1 b
								2 1 b
								3 1 b
								4 1 b
								5 1 b
								6 1 b
								7 1 b
								2 2 a
								3 2 a
								4 2 a
								5 2 b
								6 2 b
								7 2 b
								3 3 b
								4 3 b
								5 3 b
								6 3 b
								7 3 b
								4 4 a
								5 4 a
								6 4 b
								7 4 b
								5 5 b
								6 5 b
								7 5 b
								6 6 a
								7 6 b
								7 7 b
								    };
							\end{axis}
						\end{tikzpicture}
						\caption{
							Repeating the argument for Figure \ref{fig_kl_2_odd} from \eqref{eqn_k_ell_implications}, we obtain more indices $(k,\ell)$ for with $P_{k\ell}=0$ must hold.
						}
						\label{fig_kl_3_odd}
					\end{center}
				\end{figure}
			\end{minipage}
			\hfill
			\begin{minipage}{0.45\textwidth}
				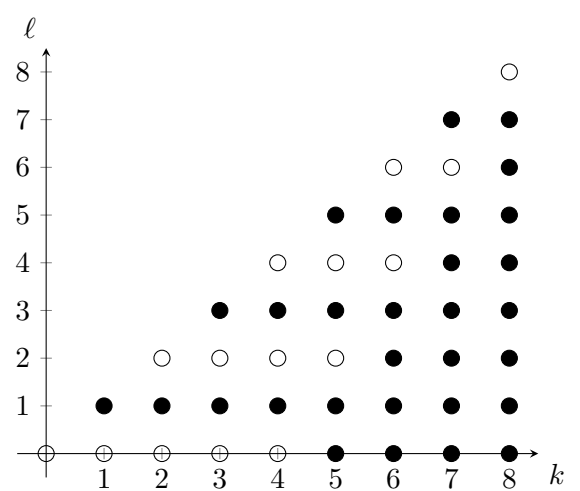
\begin{figure}[H]
					\begin{center}
						\begin{tikzpicture}
							\begin{axis}[
								axis x line=center,
								axis y line=center,
								xtick={0,1,...,8},
								ytick={0,1,...,8},
								xlabel={$k$},
								ylabel={$\ell$},
								xlabel style={below right},
								ylabel style={above left},
								xmin=-0.5,
								xmax=8.5,
								ymin=-0.5,
								ymax=8.5,
								scatter/classes={
									a={mark=o,draw=black,mark size=3pt},b={mark=*,draw=black,mark size=3pt}
								}
							]
							\addplot[scatter,only marks,scatter src=explicit symbolic]
							table[meta=label]{
								x y label
								0 0 a
								1 0 a
								2 0 a
								3 0 a
								4 0 a
								5 0 b
								6 0 b
								7 0 b
								8 0 b
								1 1 b
								2 1 b
								3 1 b
								4 1 b
								5 1 b
								6 1 b
								7 1 b
								8 1 b
								2 2 a
								3 2 a
								4 2 a
								5 2 a
								6 2 b
								7 2 b
								8 2 b
								3 3 b
								4 3 b
								5 3 b
								6 3 b
								7 3 b
								8 3 b
								4 4 a
								5 4 a
								6 4 a
								7 4 b
								8 4 b
								5 5 b
								6 5 b
								7 5 b
								8 5 b
								6 6 a
								7 6 a
								8 6 b
								7 7 b
								8 7 b
								8 8 a
   									};
							\end{axis}
						\end{tikzpicture}
						\caption{
							Analogous to Figure \ref{fig_kl_3_odd} and $\tau$ odd, we obtain vanishing conditions for $P_{k\ell}$ for $\tau$ even.
						}
						\label{fig_kl_3_even}
					\end{center}
				\end{figure}
			\end{minipage}
			\\
			\noindent
			The remaining sets of coefficients for $k$ even, $0\leq k\leq \tau$,
				\begin{equation*}
					\left\{P_{(k-i)\,(k-2i)}\ \left|\ 0\leq i\leq \frac{k}{2}\right.\right\},
				\end{equation*}
			are each determined by the choice of the value $P_{kk}\in\mathbb{R}$, respectively. Hence, the number of free variables in $h$ for it to be invariant under $a_1$ is $1+\lfloor\frac{\tau}{2}\rfloor$. Now, instead of studying when $h$ is invariant under $B$ of the form \eqref{eqn_B_affine}, observe that we can diagonalise $B$ via $C:=\left(\begin{smallmatrix} 1 & 0 & \frac{c}{2}\\ 0 & 1 & 0\\ 0 & 0 & 1 \end{smallmatrix}\right)$ and obtain, independent of $c\in\mathbb{R}$,
				\begin{equation*}
					C^{-1} B C=\left(\begin{smallmatrix}
						b-1 & 0 & 0\\
						0 & b & 0\\
						0 & 0 & b+1
					\end{smallmatrix}\right),\quad C^{-1} a_1 C =a_1.
				\end{equation*}
			Thus we can change basis so that $c=0$ in $B$. Note that while $a_1$ does not preserve the eigenspaces of $B$ spanned by the column vectors of $C$, we find that $a_1$ instead either maps them to zero, or shifts eigenspaces of $B$ to the previous one when numbering by column of $C$. Now for $h$ to be invariant under $B$ for $c=0$, we note that for any monomial $M$, $\D M\left(B\cdot\left(\begin{smallmatrix}x\\y\\z\end{smallmatrix}\right)\right)\in\mathbb{R}\cdot M$, implying that for $M=z^{\tau-k}x^{k-\ell}y^\ell$, $0\leq k\leq \tau$, $0\leq k\leq \ell$, $M$ is invariant under $B$ if and only if
				\begin{equation*}
					b=b(k,\ell)= \frac{-\tau+2k-\ell}{\tau}.
				\end{equation*}
			Hence, $h$ is invariant under $B$ if and only if all of its monomials are invariant under $B$.
			The condition $b\in(-1,1)$ thus implies $(k,\ell)\ne (0,0)$ and $(k,\ell)\ne (\tau,0)$. We further find that $b(k,\ell)=b(K,L)$ if and only if $(K,L)$ can be written as
				\begin{equation*}
					(K,L)=(k+m,\ell+2m)
				\end{equation*}
			for some $m\in\mathbb{Z}$. Hence, $h$ is invariant under $B$ if and only if it is of the form
				\begin{equation}\label{eqn_h_k_ell_final_Ps}
					h=\sum\limits_{m=0}^{\mathrm{min}\{k,\tau-k,\lfloor\frac{\tau}{2}\rfloor\}} P_{(k+m)\,2m}z^{\tau-k-m}x^{k-m}y^{2m}
				\end{equation}
			with $k\in\{1,\ldots,\tau-1\}$ fixed and $P_{(k+m)\,2m}\in\mathbb{R}$ for all $0\leq m\leq \mathrm{min}\{k,\tau-k,\lfloor\frac{\tau}{2}\rfloor\}$. See Figures \ref{fig_kl_4_odd}, \ref{fig_kl_4_even}. for a visual representation of the potentially non-zero coefficients $P_{k\ell}$ for two examples of $h$.\\
			\begin{minipage}{0.45\textwidth}
				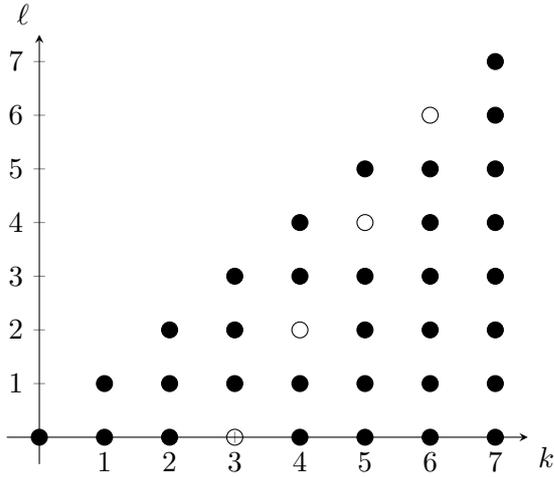
\begin{figure}[H]
					\begin{center}
						\begin{tikzpicture}
							\begin{axis}[
								axis x line=center,
								axis y line=center,
								xtick={0,1,...,7},
								ytick={0,1,...,7},
								xlabel={$k$},
								ylabel={$\ell$},
								xlabel style={below right},
								ylabel style={above left},
								xmin=-0.5,
								xmax=7.5,
								ymin=-0.5,
								ymax=7.5,
								scatter/classes={
									a={mark=o,draw=black,mark size=3pt},b={mark=*,draw=black,mark size=3pt}
								}
							]
							\addplot[scatter,only marks,scatter src=explicit symbolic]
							table[meta=label]{
								x y label
								0 0 b
								1 0 b
								2 0 b
								3 0 a
								4 0 b
								5 0 b
								6 0 b
								7 0 b
								1 1 b
								2 1 b
								3 1 b
								4 1 b
								5 1 b
								6 1 b
								7 1 b
								2 2 b
								3 2 b
								4 2 a
								5 2 b
								6 2 b
								7 2 b
								3 3 b
								4 3 b
								5 3 b
								6 3 b
								7 3 b
								4 4 b
								5 4 a
								6 4 b
								7 4 b
								5 5 b
								6 5 b
								7 5 b
								6 6 a
								7 6 b
								7 7 b
								    };
							\end{axis}
						\end{tikzpicture}
						\caption{
							The potentially non-vanishing coefficients $P_{k\ell}$ in $h$ \eqref{eqn_h_k_ell_final_Ps} for $\tau=7$, $k=3$, marked with white dots.
						}
						\label{fig_kl_4_odd}
					\end{center}
				\end{figure}
			\end{minipage}
			\hfill
			\begin{minipage}{0.45\textwidth}
				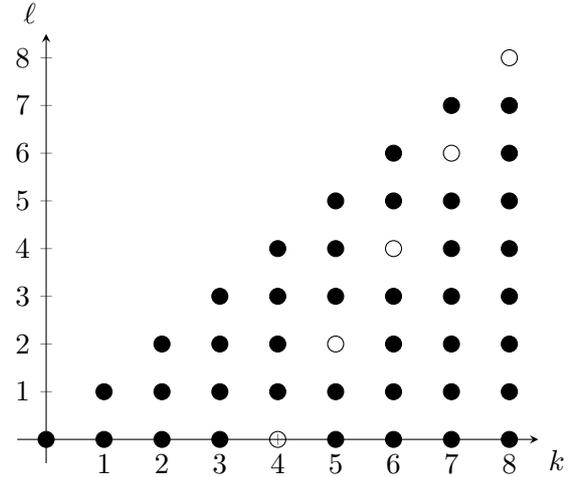
\begin{figure}[H]
					\begin{center}
						\begin{tikzpicture}
							\begin{axis}[
								axis x line=center,
								axis y line=center,
								xtick={0,1,...,8},
								ytick={0,1,...,8},
								xlabel={$k$},
								ylabel={$\ell$},
								xlabel style={below right},
								ylabel style={above left},
								xmin=-0.5,
								xmax=8.5,
								ymin=-0.5,
								ymax=8.5,
								scatter/classes={
									a={mark=o,draw=black,mark size=3pt},b={mark=*,draw=black,mark size=3pt}
								}
							]
							\addplot[scatter,only marks,scatter src=explicit symbolic]
							table[meta=label]{
								x y label
								0 0 b
								1 0 b
								2 0 b
								3 0 b
								4 0 a
								5 0 b
								6 0 b
								7 0 b
								8 0 b
								1 1 b
								2 1 b
								3 1 b
								4 1 b
								5 1 b
								6 1 b
								7 1 b
								8 1 b
								2 2 b
								3 2 b
								4 2 b
								5 2 a
								6 2 b
								7 2 b
								8 2 b
								3 3 b
								4 3 b
								5 3 b
								6 3 b
								7 3 b
								8 3 b
								4 4 b
								5 4 b
								6 4 a
								7 4 b
								8 4 b
								5 5 b
								6 5 b
								7 5 b
								8 5 b
								6 6 b
								7 6 a
								8 6 b
								7 7 b
								8 7 b
								8 8 a
   									};
							\end{axis}
						\end{tikzpicture}
						\caption{
							The analogue picture of Figure \ref{fig_kl_4_odd} for $\tau=8$, $k=4$.
						}
						\label{fig_kl_4_even}
					\end{center}
				\end{figure}
			\end{minipage}
			\\
			\noindent
			By combining our results, we find that $h$ is invariant under $a_1$ and $B$ if and only if it can be written as
				\begin{equation*}
					h=\sum\limits_{m=0}^{k} P_{(k+m)\,2m}z^{\tau-k-m}x^{k-m}y^{2m}
				\end{equation*}
			for some fixed $k\in\left\{1,\ldots,\lfloor\frac{\tau}{2}\rfloor\right\}$ with $b=\frac{-\tau+2k}{\tau}$, and $P_{(k+m)\,2m}$ fulfilling
				\begin{equation}\label{eqn_P_explicit_form_nonabelian_case}
					P_{(k+m)\,2m} = \left(-\frac{1}{2}\right)^{m} \binom{m}{k} P_{k\, 0}
				\end{equation}
			for all $1\leq m\leq k$ with $P_{k\,0}\in\mathbb{R}$ fixed. Since $1\leq k\leq \lfloor\frac{\tau}{2}\rfloor$, we might further rewrite $h$ as
				\begin{equation}\label{eqn_h_affine_initial_coords}
					h= P_{k\, 0} z^{\tau-2k}\left(zx-\tfrac{1}{2}y^2\right)^k.
				\end{equation}
			Now we need to check which of the above polynomials are hyperbolic and thereafter which are pairwise inequivalent, and we need to describe the connected components of $\mathrm{hyp}_1(h)$. By rescaling we might assume without loss of generality that $P_{k\,0}=\pm 1$, and changing coordinates $(x,y,z)\to \left(x-z,x+z,\sqrt{2}y\right)$ transforms $h$ to
				\begin{equation*}
					h= \pm (x+z)^{\tau-2k}\left(x^2-y^2-z^2\right)^k.
				\end{equation*}
			Observe that in any case
				\begin{equation*}
					\{h=0\}=\{x+z=0\}\cup\{x^2-y^2-z^2=0\},
				\end{equation*}
			meaning that the complement of $\{h=0\}$ in the ambient space $\mathbb{R}^3$ has four connected components. Only two of those are cones, namely the ones containing the points $(1,0,0)^\mathrm{T}$ and $(-1,0,0)^\mathrm{T}$. Thus, in order to check which of the above polynomials are hyperbolic, it suffices to check whether at least one of the aforementioned cones contain a hyperbolic point.
			
			We start with the case of $\tau$ being odd. Then $k<\frac{\tau}{2}$ is equivalent to $k\leq \lfloor\frac{\tau}{2}\rfloor$, and an overall sign change $(x,y,z)\to(-x,-y,-z)$ maps $h$ to $-h$, independent of $k$. Hence we might assume $P_{k\,0}=1$, so that
				\begin{equation}\label{eqn_h_affine_nostab}
					h= (x+z)^{\tau-2k}\left(x^2-y^2-z^2\right)^k.
				\end{equation}
			This further shows that if $h$ is hyperbolic, $\mathrm{hyp}_1(h)$ has one connected component. For $p=(1,0,0)^\mathrm{T}$ we then check that $h(p)=1$ and
				\begin{equation}\label{eqn_affine_hyp_check}
					-\partial^2h_p = \left(\begin{smallmatrix}
						-\tau(\tau-1) & 0 & (\tau-1)(-\tau+2k)\\
						0 & 2k & 0\\
						(\tau-1)(-\tau+2k) & 0 & -\tau(\tau-1)+4k(\tau-k)
					\end{smallmatrix}\right).
				\end{equation}
			Using $\tau\geq 3$, $1\leq k<\frac{\tau}{2}$, and $\mathrm{det}\left(-\partial^2h_p\right)=-8k^2(\tau-1)(\tau-k)<0$, we obtain that $-\partial^2h_p$ has Minkowski signature, so $p$ is in fact a hyperbolic point of $h$. The single connected component of $\mathrm{hyp}_1(h)$ is precisely given by the connected component of $\{h=1\}$ that contains $p$. This follows from the fact connected homogeneous special surfaces span a cone, and the only other potential candidate for a connected component of $\mathrm{hyp}(h)$ would have to contain $-p$. But since $\tau$ is odd, $h(-p)=-1$. The full automorphism group $G^h$ of $h$ is thus generated by the stabiliser of $p$, $\mathrm{stab}^h_p$, and the connected component of the identity transformation, $G^h_0$. The latter is generated by $a_1$ and $B$ by construction. Any element of $\mathrm{stab}^h_p$ is in particular an automorphism of the set $\{h=0\}$, and more precisely must preserve $\{x+z=0\}$ and $\{x^2-y^2-z^2=0\}$, and their intersection $\{x+z=0,\ y=0\}=\mathrm{span}\left\{\left(\begin{smallmatrix}1\\0\\-1\end{smallmatrix}\right)\right\}$. A matrix $A\in\mathrm{GL}(3)$ maps $p$ to itself and leaves $\{x+z=0\}$ and $\{x+z=0,\ y=0\}=\mathrm{span}\left\{\left(\begin{smallmatrix}1\\0\\-1\end{smallmatrix}\right)\right\}$ invariant if and only if it is of the form
				\begin{equation*}
					A=\left(\begin{matrix}
						1 & u & v\\
						0 & w & 0\\
						0 & -u & 1-v
					\end{matrix}\right)
				\end{equation*}
			for some $u,v,w\in\mathbb{R}^3$, such that $A$ is invertible. We verify that $A$ preserves $\{x^2-y^2-z^2=0\}$ if and only if $u=v=0$ and $w=\pm 1$. Hence, $\mathrm{stab}^h_p$ contains only the identity and the transformation $y\to -y$. They commute with all elements of $G^h_0$ that are generated by
				\begin{equation*}
					N^{-1}BN=\left(\begin{matrix}
						b & 0 & 1\\
						0 & b & 0\\
						1 & 0 & b
					\end{matrix}\right),
				\end{equation*}
			but do not commute with elements generated by
				\begin{equation*}
					N^{-1}a_1N=\frac{1}{\sqrt{2}}\left(\begin{matrix}
						0 & 1 & 0\\
						1 & 0 & 1\\
						0 & -1 & 0
					\end{matrix}\right),
				\end{equation*}
			where $N$ describes the coordinate transformation $(x,y,z)\to(x-z,\sqrt{2}y,x+z)$. Hence, $G^h=G^h_0\ltimes \mathbb{Z}_2$.
			
			Next let $\tau$ be even. In this case we can exclude $k=\lfloor\frac{\tau}{2}\rfloor=\frac{\tau}{2}$ since that will lead to the one-dimensional stabiliser case studied in Section \ref{sect_1dimstab}. So from here on we will assume that $1\leq k<\frac{\tau}{2}$. If $k$ is additionally even, for $\{h>0\}$ to not be empty $P_{k\,0}=1$ is a necessary requirement, meaning that $h$ is again of the form \eqref{eqn_h_affine_nostab}. In general, if $P_{k\,0}=1$, we check that $p=(1,0,0)^\mathrm{T}$ is a hyperbolic point of $h$ just as we did for $\tau$ odd. In this case, $\mathrm{hyp}_1(h)$ has two isometric connected components that are related by $(x,y,z)\to (-x,-y,-z)$. The negative identity transformation commutes with any element in $G^h$, and we find similar to the case $\tau$ odd that $G^h\cong \left(G^h_0\ltimes \mathbb{Z}_2\right)\times\mathbb{Z}_2$. If $\tau$ is even and $k$ is odd, we obtain the same result as for $\tau$ even and $k$ even if $P_{k\,0}=1$. The missing case we need to study is thus $\tau$ even, $k$ odd, and $P_{k\,0}=-1$. The two connected components of the complement of $\{h=0\}$ that are cones contain the points $p$ and $-p$, respectively. In both cases however, $h(p)=h(-p)=-1$. Hence, neither cone contains hyperbolic points of $h$, and by the supposed homogeneity and consequent property that connected components of $\mathrm{hyp}_1(h)$ are closed in the ambient space \cite[Prop.\,1.8]{CNS}, we obtain that $\{h=1\}$ does not contain a special homogeneous surface.
			
			It remains to show that the polynomials of the form $h=(x+z)^{\tau-2k}\left(x^2-y^2-z^2\right)^k$, $1\leq k< \frac{\tau}{2}$, are pairwise inequivalent. To see this, recall that a supposed linear transformation mapping $h$ corresponding to $k$ to $\overline{h}$ corresponding to $\overline{k}\ne k$ must restrict to an automorphism of the vector subspace $\{x+z=0\}\subset\mathbb{R}^3$ that furthermore restricts to an automorphism of $\mathrm{span}\left\{\left(\begin{smallmatrix}1\\ 0\\ -1\end{smallmatrix}\right)\right\}$. Observe that $h$ has a zero of degree $\tau-2k$ in every point $\{x+z=0\}\cap\{y\ne0\}$, but $\overline{h}$ has a zero of degree $\tau-2\overline{k}\ne \tau-2k$ in every such point. This excludes $h$ and $\overline{h}$ being equivalent.
			
			In order to determine the scalar curvature of the connected components of $\mathrm{hyp}_1(h)$ with respect to the centro-affine metric $g$, we will construct a left-invariant orthonormal frame $\{X,Y\}$ of $T\mathcal{H}$. We will work in our initial linear coordinates, so that
				\begin{equation*}\label{eqn_h_affine_initial_coords_P0k_1}
					h=z^{\tau-2k}\left(xz-\frac{1}{2}y^2\right)^k
				\end{equation*}
			as in \eqref{eqn_h_affine_initial_coords}. In these coordinates, it follows from \eqref{eqn_affine_hyp_check} that the point $q:=(1,0,1)^\mathrm{T}$ with $h(q)=1$ is hyperbolic. We calculate
				\begin{equation*}
					a_1q=\left(\begin{smallmatrix}0\\1\\0\end{smallmatrix}\right),\quad Bq=\left(\begin{smallmatrix}b-1\\0\\b+1\end{smallmatrix}\right),
				\end{equation*}
			where we recall that $b=\frac{-\tau+2k}{\tau}$, and verify that $-\partial^2h_q(a_1q,Bq)=0$. We further check that
				\begin{equation*}
					g_q(a_1q,a_1q)=\frac{k}{\tau},\quad g_q(Bq,Bq)=\frac{4k(\tau-k)}{\tau^2},
				\end{equation*}
			and thus define $X$ and $Y$ to be the left-invariant extensions of $X_q:=\sqrt{\frac{\tau}{k}}a_1q$ and $Y_q:=\frac{\tau}{2\sqrt{k(\tau-k)}}Bq$, respectively. For their Lie bracket we obtain using $[a_1,B]=a_1$ that
				\begin{equation*}
					[X,Y]=\frac{\tau}{2\sqrt{k(\tau-k)}}X.
				\end{equation*}
			Using Koszul's formula, we find for the Levi-Civita connection $\nabla$ of $(\mathcal{H},g)$
				\begin{equation*}
					\nabla_X X= -\frac{\tau}{2\sqrt{k(\tau-k)}}Y,\quad \nabla_X Y= \frac{\tau}{2\sqrt{k(\tau-k)}}X,\quad \nabla_Y X= 0,\quad \nabla_Y Y= 0.
				\end{equation*}
			Hence, the non-vanishing terms of the Riemannian curvature tensor $R$ are determined by
				\begin{equation*}
					R(X,Y)X=\frac{\tau^2}{4k(\tau-k)}Y,\quad R(X,Y)Y=-\frac{\tau^2}{4k(\tau-k)}X,
				\end{equation*}
			and the scalar curvature of $(\mathcal{H},g)$ is given by
				\begin{equation*}
					S=-\frac{\tau^2}{2k(\tau-k)}.
				\end{equation*}
			Note that this result coincides with \cite[Prop.\,5.12]{L1} for the special case $\tau=3$, $k=1$. We further observe that the assignment $k\mapsto S$ for $1\leq k\leq \left\lfloor\frac{\tau}{2}\right\rfloor$ is injective, which follows from $\partial_k(k(\tau-k))>0$ for all $k\in\left(0,\left\lfloor\frac{\tau}{2}\right\rfloor\right)$. This also presents an alternative way to see that the polynomials $h$ in \eqref{eqn_h_affine_initial_coords_P0k_1} are inequivalent for different values of $k$.
			
			Lastly, we study the case $a_6$. The equation $[a_6,B]=a_6$ is equivalent to $B$ being of the form
				\begin{equation*}
					B = \left(\begin{matrix}
						B_{11} & -B_{12} & 0\\
						B_{12} & B_{11}-1 & 0\\
						0 & 0 & B_{33}
					\end{matrix}\right).
				\end{equation*}
			After adding $-B_{12}a_6$ to $B$, we might assume that $B_{12}=0$. If $B_{33}=0$, both the induced actions of $a_6$ and $B$ restrict to any plane orthogonal to $(0,0,1)^\mathrm{T}$, which contradicts the hyperbolicity of the considered polynomial. Hence, we can assume without loss of generality that $B$ is given by
				\begin{equation*}
					B=\left(\begin{matrix}
						b & 0 & 0\\
						0 & b-1 & 0\\
						0 & 0 & b+c
					\end{matrix}\right),\quad b,c\in\mathbb{R},
				\end{equation*}
			so that the set $\{b,b-1,b+c\}$ contains at least one negative and at least one positive number, which follows from Lemma \ref{lem_H_bounded_away_from_0}. When studying one-dimensional stabilisers, we have seen that $h$ being invariant under $a_6$ requires it to be of the form \eqref{eqn_polys_onedim_stabiliser}. We now calculate and find that $\D h_{\left(\begin{smallmatrix}x\\y\\z\end{smallmatrix}\right)}\left(B\cdot\left(\begin{smallmatrix}x\\y\\z\end{smallmatrix}\right)\right)=0$ is equivalent to
				\begin{equation*}
					c_{\ell}\left((b\tau+c(\tau-2\ell))x^2 + (b\tau-2\ell+c(\tau-2\ell))y^2\right)=0
				\end{equation*}
			for all $0\leq \ell \leq \left\lfloor\frac{\tau}{2}\right\rfloor$, where $c_\ell$ is used as in equation \eqref{eqn_polys_onedim_stabiliser}. Hence the only potentially non-zero coefficient $c_\ell$ is $c_0$. But $h=c_0z^\tau$ is not a hyperbolic polynomial. This excludes the generator $a_6$ and finishes the case of $\mathfrak{g}$ being non-abelian.
			
		\subsubsection{$\mathfrak{g}$ abelian}
			Now suppose that $\mathfrak{g}$ is abelian. For $B=(B_{ij})\in\mathfrak{gl}(3)$ we have to study $[a_i,B]=0$ for all $1\leq i\leq 6$ and then determine which of these pairs $(a_i,B)$ are linearly independent. In the following, $(x,y,z)$ denote linear coordinates on $\mathbb{R}^3$. We will first classify all two-dimensional abelian Lie subalgebras of $\mathfrak{gl}(3)$ that potentially act transitively on a special homogeneous surface, and in the second step classify the respective hyperbolic polynomials.
			
			For $a_1$, we obtain $[a_1,B]=0$ if and only if $B$ is of the form
				\begin{equation*}
					B=\left(\begin{matrix}
						B_{11} & B_{12} & B_{13}\\
						0 & B_{11} & B_{12}\\
						0 & 0 & B_{11}
					\end{matrix}\right).
				\end{equation*}
			Hence, $B$ has one eigenvalue, namely $B_{11}$. For $B_{11}\ne 0$, the origin in in the closure of every orbit of the induced action of $B$ and this case can thus be excluded by Lemma \ref{lem_H_bounded_away_from_0}. For $B_{11}=0$ and $B\ne0$, $a_1$ and $B$ being linearly independent implies that $B$ must be equivalent to $\left(\begin{smallmatrix} 0 & 1 & 0\\ 0 & 0 & 0\\ 0 & 0 & 0 \end{smallmatrix}\right)$. This case is excluded by Lemma \ref{lem_no_rays}. Hence, there exists no two-dimensional abelian Lie subalgebra of $\mathfrak{gl}(3)$ acting on a special homogeneous surface with one generator being $a_1$.
			
			For $a_2$, $[a_2,B]=0$ requires $B$ to be of the form
				\begin{equation}\label{eqn_B_form_a_2}
					B=\left(\begin{matrix}
						B_{11} & B_{12} & 0\\
						0 & B_{11} & 0\\
						0 & 0 & B_{33}
					\end{matrix}\right).
				\end{equation}
			We thus obtain three cases up to scale for $B$, depending on the sign of $B_{11}$ and $B_{12}$,
				\begin{equation*}
					b_{21}=\left(\begin{smallmatrix}
						1 & 0 & 0\\
						0 & 1 & 0\\
						0 & 0 & b
					\end{smallmatrix}\right),\ b< 0,\quad
					b_{22}=\left(\begin{smallmatrix}
						1 & 1 & 0\\
						0 & 1 & 0\\
						0 & 0 & b
					\end{smallmatrix}\right),\ b< 0,\quad
					b_{23}=\left(\begin{smallmatrix}
						0 & 1 & 0\\
						0 & 0 & 0\\
						0 & 0 & b
					\end{smallmatrix}\right),\ b\ne 0.
				\end{equation*}
			The respective conditions for $b$ are, again, implied by Lemmas \ref{lem_no_rays} and \ref{lem_H_bounded_away_from_0}. Independent of the choice for $b$ in $b_{23}$, that case can be excluded by noting that either $a_2$ and $b_{23}$ are linearly independent, or the studied hyperbolic polynomial must be invariant under $a_2-b_{23}=\left(\begin{smallmatrix} 0 & 0 & 0\\ 0 & 0 & 0\\ 0 & 0 & \mu-b\end{smallmatrix}\right)$, which contradicts Lemma \ref{lem_no_rays}. Furthermore, after replacing $b_{22}\to b_{22}-a_2$, we are left with only one case, namely $b_{21}$. This is a special case for the first generator being $a_4$ which we will keep in mind. Since $b<0$ in the definition of $b_{21}$, $h$ is invariant under $a_2-\frac{\mu}{b}b_{21}$, which is up to scale equivalent to one of
				\begin{equation}\label{eqn_a2_three_cases}
					\left(\begin{smallmatrix}
						1 & 0 & 0\\
						0 & 1 & 0\\
						0 & 0 & 0
					\end{smallmatrix}\right),\quad
					\left(\begin{smallmatrix}
						1 & 1 & 0\\
						0 & 1 & 0\\
						0 & 0 & 0
					\end{smallmatrix}\right),\quad
					\left(\begin{smallmatrix}
						0 & 1 & 0\\
						0 & 0 & 0\\
						0 & 0 & 0
					\end{smallmatrix}\right).
				\end{equation}
			But all of the above cases are in contradiction to Lemma \ref{lem_no_rays} or Lemma \ref{lem_H_bounded_away_from_0}. We conclude that there exists no two-dimensional abelian Lie subalgebra of $\mathfrak{gl}(3)$ acting on a special homogeneous surface with one generator being $a_2$.
			
			Next consider $a_3$. For $B$ to fulfil $[a_3,B]=0$, it must, as for $a_2$, be of the form \eqref{eqn_B_form_a_2}. With an analogous argument as for the case $a_2$ we obtain that, up to a linear transformation, the corresponding hyperbolic polynomial would be invariant under one of the matrices in \eqref{eqn_a2_three_cases}, contradicting Lemma \ref{lem_no_rays} or Lemma \ref{lem_H_bounded_away_from_0}. Hence, we can also exclude $a_3$ as the first generator for our desired abelian Lie subalgebra of $\mathfrak{gl}(3)$.
			
			Next, consider the generator $a_4$. Then $[a_4,B]=0$ is equivalent to $B$ being of the form
				\begin{equation*}
					B=\left(\begin{matrix}
						B_{11} & 0 & 0\\
						0 & B_{22} & 0\\
						0 & 0 & B_{33}
					\end{matrix}\right)
				\end{equation*}
			for $\lambda\in\mathbb{R}_{\geq 0}\setminus\{1\}$, and of the form
				\begin{equation*}
					B=\left(\begin{matrix}
						B_{11} & B_{12} & 0\\
						B_{21} & B_{22} & 0\\
						0 & 0 & B_{33}
					\end{matrix}\right)
				\end{equation*}
			for $\lambda=1$. We start with the case $\lambda\in\mathbb{R}_{\geq 0}\setminus\{1\}$. After changing $B\to B-B_{11}a_4$, rescaling, and using either Lemma \ref{lem_no_rays} or Lemma \ref{lem_H_bounded_away_from_0}, we might assume that $B$ is of the form
				\begin{equation*}
					B=\left(\begin{smallmatrix}
						0 & 0 & 0\\
						0 & 1 & 0\\
						0 & 0 & b
					\end{smallmatrix}\right),\ b<0.
				\end{equation*}
			Note this excludes the case of the generators being $a_4$ and $a_2$ as discussed above. For $\lambda\in\mathbb{R}_{\geq 0}\setminus\{1\}$ we might use alternative generators for the spanned abelian Lie subalgebra of $a_4$ and $B$ by replacing $a_4\to a_4-\lambda B$, namely $B$ and $A:=\left(\begin{smallmatrix} 1 & 0 & 0\\ 0 & 0 & 0\\ 0 & 0 & a\end{smallmatrix}\right)$ for some $a<0$. The latter condition comes, again, from Lemma \ref{lem_no_rays} and Lemma \ref{lem_H_bounded_away_from_0}. For $\lambda=0$, we can simply relabel $\mu=a$ and obtain the same set of generators.
			
			Next, consider $\lambda=1$. Since for $\lambda=1$, $a_4$ is invariant under the adjoint action of $\mathrm{GL}(2)\times\mathrm{Id}_{z}$ and under switching the variables $x$ and $y$ of the ambient space, we can, after rescaling, assume without loss of generality that $B$ is either diagonal, or of one the forms
				\begin{equation*}
					b_{41}=\left(\begin{smallmatrix}
						1 & 1 & 0\\
						0 & 1 & 0\\
						0 & 0 & b
					\end{smallmatrix}\right),\ b< 0,\quad
					b_{42}=\left(\begin{smallmatrix}
						0 & 1 & 0\\
						0 & 0 & 0\\
						0 & 0 & b
					\end{smallmatrix}\right),\ b\ne 0,\quad
					b_{43}=\left(\begin{smallmatrix}
						1 & -c & 0\\
						c & 1 & 0\\
						0 & 0 & b
					\end{smallmatrix}\right),\ c>0,\ b<0.
				\end{equation*}
			Note that we have again used Lemmas \ref{lem_no_rays} and \ref{lem_H_bounded_away_from_0} to obtain the respective ranges for $b$ and $c$ in the above equation. The diagonal case has already been covered in our study of the related $\lambda\in\mathbb{R}_{\geq 0}\setminus\{1\}$. Since, up to relabelling of the variables, $b_{41}-a_4=a_3$ for $\lambda=1$, and that case has already been shown not to admit an abelian Lie subalgebra fulfilling our hyperbolicity conditions, we can exclude the second generator being $b_{41}$. Similarly, the case $b_{42}$ can be excluded. For $b_{43}$, we can replace $b_{43}\to b_{43}-a_4$ and after rescaling obtain the generator $B:=\left(\begin{smallmatrix} 0 & -1 & 0\\ 1 & 0 & 0\\ 0 & 0 & b\end{smallmatrix}\right)$ with $b<0$. The abelian Lie subalgebra $\mathrm{span}\{a_4,B\}$ of $\mathfrak{gl}(3)$ is also a special case of the upcoming case with first generator being $a_5$.
			
			Suppose lastly that the first generator is either $a_5$ or $a_6$. For $[a_5,B]=0$, respectively $[a_6,B]=0$, to hold, the upper left $2\times2$-block of $B$ must correspond to a complex number, that is $B$ must be of the form
				\begin{equation*}
					B=\left(\begin{matrix}
						B_{11} & B_{12} & 0\\
						-B_{12} & B_{11} & 0\\
						0 & 0 & B_{33}
					\end{matrix}\right).
				\end{equation*}
			For $B_{12}=0$, it follows that $B_{11}\ne 0$ by either Lemma \ref{lem_no_rays} or Lemma \ref{lem_H_bounded_away_from_0}, and we are in one of the cases with first generator being $a_4$ studied above. For $B_{12}\ne 0$, we might replace $B\to B+B_{12}a_5$, respectively $B\to B+B_{12}a_6$, and are precisely in the latter case again.
			
			Summarising, we have shown the following. Up to isomorphism, there exist two two-parameter families of abelian Lie subalgebras of $\mathfrak{gl}(3)$ that at this point can not be excluded to act transitively on a special homogeneous surface, namely
				\begin{align}
					&\mathfrak{g}_1:=\mathrm{span}\left\{A_1,B_1\right\},\ A_1=\left(\begin{smallmatrix} 1 & 0 & 0\\ 0 & 0 & 0\\ 0 & 0 & a\end{smallmatrix}\right),\ B_1=\left(\begin{smallmatrix} 0 & 0 & 0\\ 0 & 1 & 0\\ 0 & 0 & b\end{smallmatrix}\right),\ a<0,\ b<0,\label{eqn_g_1}\\
					&\mathfrak{g}_2:=\mathrm{span}\left\{A_2,B_2\right\},\ A_2=\left(\begin{smallmatrix} 1 & 0 & 0\\ 0 & 1 & 0\\ 0 & 0 & a\end{smallmatrix}\right),\ B_2=\left(\begin{smallmatrix} 0 & -1 & 0\\ 1 & 0 & 0\\ 0 & 0 & b\end{smallmatrix}\right),\ a<0,\ b<0.\notag
				\end{align}
			For both families we now need to determine the values of $a$ and $b$, so that there exists a hyperbolic polynomial that is invariant under $\mathfrak{g}_1$ or $\mathfrak{g}_2$, respectively. We then need to classify all such hyperbolic polynomials up to equivalence.
			
			We start with $\mathfrak{g}_1$. Note that for all monomials $M$ in $(x,y,z)$, both $\D M \left(A_1\cdot\left(\begin{smallmatrix}x\\y\\z\end{smallmatrix}\right)\right)$ and $\D M \left(B_1\cdot\left(\begin{smallmatrix}x\\y\\z\end{smallmatrix}\right)\right)$ are contained in $\mathbb{R}\cdot M$. Hence, a homogeneous polynomial $h$ of degree $\tau\geq 3$ is invariant under $\mathfrak{g}_1$ if and only if all its monomial terms are. Let $M=m z^{\tau-k}x^{k-\ell}y^\ell$ for some $m\ne0$, $k\in\{0,\ldots,\tau\}$, $\ell\in\{0,\ldots,k\}$. We find that the monomial $M$ is invariant under $\mathfrak{g}_1$ if and only if
				\begin{equation}\label{eqn_g1_a_b}
					a(\tau-k)+k-\ell=0,\quad b(\tau-k)+\ell=0.
				\end{equation}
			The above two equations do not admit a solution with $k=0$, since then by assumption $\ell=0$ and consequently $a=b=0$, contradicting the definition of $\mathfrak{g}_1$. For $k=\tau$, the second equation implies $\ell=0$, thereby contradicting the first of the two equations. Hence, \eqref{eqn_g1_a_b} has a solution if and only $k\in\{1,\ldots,\tau-1\}$, which is given by
				\begin{equation}\label{eqn_g1_a_b_explicit}
					a=\frac{\ell-k}{\tau-k},\quad b=-\frac{\ell}{\tau-k}.
				\end{equation}
			Note that $a<0$ and $b<0$ imply that $\ell\in\{1,\ldots,k-1\}$, which allows us to further restrict the range of $k$ to $\{2,\ldots,\tau-1\}$. Assume that $a$ and $b$ are as above for a chosen triple of $(\tau,k,\ell)$. We will show that for $\tau$ fixed there exist no other such triple leading to the same values of $a$ and $b$. Suppose that
				\begin{equation}\label{eqn_g1_a_b_uniqueness}
					\frac{\ell-k}{\tau-k}=\frac{L-K}{\tau-K},\quad -\frac{\ell}{\tau-k}=-\frac{L}{\tau-K},
				\end{equation}
			for triples $(\tau,k,\ell),(\tau,K,L)\in\mathbb{N}_{\geq 3}\times\{2,\ldots,\tau-1\}\times\{1,\ldots,k-1\}$. By substituting the second equation in \eqref{eqn_g1_a_b_uniqueness} into the first, we obtain $-\frac{k}{\tau-k}=-\frac{K}{\tau-K}$ which is easily shown to hold if and only if $-\frac{\tau}{\tau-k}=-\frac{\tau}{\tau-K}$. Hence, $k=K$, and this implies $\ell=L$. Summarising, we have shown that a homogeneous polynomial $h$ is invariant under $\mathfrak{g}_1$ if and only if $a$ and $b$ fulfil \eqref{eqn_g1_a_b_explicit} and, up to positive rescaling, $h=\pm z^{\tau-k}x^{k-\ell}y^\ell$ for a choice of $(\tau,k,\ell)\in\mathbb{N}_{\geq 3}\times\{2,\ldots,\tau-1\}\times\{1,\ldots,k-1\}$. We can further restrict to
				\begin{equation*}
					h=z^{\tau-k}x^{k-\ell}y^\ell,
				\end{equation*}
			either by $(x,y,z)\to (-x,-y,-z)$, or by noting that if $\tau-k$, $k-\ell$, and $\ell$ are all even, $-z^{\tau-k}x^{k-\ell}y^\ell$ is non-positive, thus excluding hyperbolicity. We now need to study which of the polynomials of the form $h=z^{\tau-k}x^{k-\ell}y^\ell$ are hyperbolic. It will turn out that, in fact, all of them are. To show this it suffices that in each case, $p=(1,1,1)^\mathrm{T}$ is a hyperbolic point of $h$. The latter point it contained in $\{h=1\}$, and we further know by Euler's theorem for homogeneous functions that
				\begin{equation*}
					-\partial^2 h_{p}\left(p,p\right)=-\tau(\tau-1)<0.
				\end{equation*}
			Recall that $T_{p}\{h=1\}$ is orthogonal to $\mathbb{R}\cdot p$ with respect to $-\partial^2h_{p}$. Thus, in order to show that $p$ is a hyperbolic point it suffices to show that $\det\left(-\partial^2h_{p}\right)<0$ and that for some $v\in T_{p}\{h=1\}=\ker(\D h_p)$ we have $-\partial^2h_{p}(v,v)>0$. We check that
				\begin{equation*}
					\det\left(-\partial^2h_{p}\right)=-\ell(\tau-1)(\tau-k)(k-\ell),
				\end{equation*}
			which is negative for the allowed values of $\tau,k,\ell$.
			We further calculate $\D h_p=(k-\ell)\D x + \ell\D y + (\tau-k)\D z$, and choose $v=(0, \tau-k, -\ell)^\mathrm{T}$. We obtain
				\begin{equation*}
					-\partial^2 h_p(v,v)=\ell(\tau-k)(\tau-k+\ell)>0
				\end{equation*}
			as required. Hence, $h$ is indeed hyperbolic.
			
			It remains to describe the connected components of $\mathrm{hyp}_1(h)$ and the automorphism group $G^h$ for each allowed triple $(\tau,k,\ell)$. Note that $\{h=0\}=\{x=0\}\cup\{y=0\}\cup\{z=0\}$, meaning that $\{h\ne 0\}$ has eight connected components. For $\tau$ odd, either $1$ or $3$ out of the exponents $\tau-k,k-\ell,\ell$ must be odd. In any of these cases, we see that $\{h>0\}$ has four connected components and, in fact, coincides with $\mathrm{hyp}_1(h)$. The connected components are pairwise isometric via combinations of sign flips in the coordinates $x,y,z$. For $\tau$ even, either $0$ or $2$ out of the exponents $\tau-k,k-\ell,\ell$ must be odd. In the first case we find that $\mathrm{hyp}_1(h)$ has eight isometric connected components, and in the latter case $\mathrm{hyp}_1(h)$ again has four isometric connected components. In order to find the automorphism group $G^h$ of $h$ we need to find the stabiliser of $p$ and determine which of its elements are not obtained from combinations of sign flips in the coordinates $x,y,z$. To do so, we write
				\begin{equation}\label{eqn_h_abc_form}
					h=x^a y^b z^c,\quad a,b,c\in\mathbb{N},\quad a+b+c=\tau,\quad a\leq b\leq c.
				\end{equation}
			Every element $s$ of $\mathrm{stab}^h_p$ must preserve the union of the pairwise intersections of $\{x=0\}$, $\{y=0\}$, and $\{z=0\}$, since $s$ is in particular a continuous map. Since $s$ is linear, it can thus be written as the composition of a diagonal linear transformation and a permutation of coordinates. Since any variable switch maps $p$ to itself and $p$ is by assumption an eigenvector of $s$ with eigenvalue $1$, we find that $s$ must be, up to switching coordinates, the identity transformation. Thus, $\mathrm{stab}^h_p$ is either trivial if $a<b<c$, isomorphic to $\mathbb{Z}_2$ if $a\leq b<c$ or $a<b\leq c$, or isomorphic to the group of permutations of a set with three elements $\sigma_3$. Since no combinations of sign flips in $x,y,z$ that is not the identity preserves any connected component of $\mathrm{hyp}_1(h)$ and elements of $\mathfrak{g}_1$ are diagonal, we find that for $a,b,c$ even
				\begin{equation}\label{eqn_Gh_flat_case_even}
					G^h\cong\left\{\begin{tabular}{ll}
						$a<b<c$: & $\mathbb{R}^2\times\mathbb{Z}_2^3$,\\
						$a=b<c$ or $a<b=c$: & $(\mathbb{R}^2\times\mathbb{Z}_2^3)\ltimes\mathbb{Z}_2$,\\
						$a=b=c$: & $(\mathbb{R}^2\times\mathbb{Z}_2^3)\ltimes\sigma_3$,
					\end{tabular}\right.
				\end{equation}
			and for at least one of $a,b,c$ odd
				\begin{equation}\label{eqn_Gh_flat_case_odd}
					G^h\cong\left\{\begin{tabular}{ll}
						$a<b<c$: & $\mathbb{R}^2\times\mathbb{Z}_2^2$,\\
						$a=b<c$ or $a<b=c$: & $(\mathbb{R}^2\times\mathbb{Z}_2^2)\ltimes\mathbb{Z}_2$,\\
						$a=b=c$: & $(\mathbb{R}^2\times\mathbb{Z}_2^2)\ltimes\sigma_3$.
					\end{tabular}\right.
				\end{equation}
				
			It remains to show that the polynomials $h$ as in \eqref{eqn_h_abc_form} are pairwise inequivalent. Similarly to the calculation of the stabiliser of $p$ above, we note that any supposed transformation relating two such polynomials must be compositions of switching variables and a diagonal transformation. The condition $a\leq b\leq c$ thus excludes that two distinct polynomials of the form \eqref{eqn_h_abc_form} are equivalent.
			
			Next consider $\mathfrak{g}_2$ and suppose that a hyperbolic polynomial $h$ is invariant under it. Then $h$ is in particular invariant under $C:=B_2-\frac{b}{a}A_2$, which is of the form
				\begin{equation*}
					C=\left(\begin{smallmatrix}
						c & -1 & 0\\
						1 & c & 0\\
						0 & 0 & 0
					\end{smallmatrix}\right),\quad c<0.
				\end{equation*}
			Hyperbolicity is an open condition and, hence, there exists a hyperbolic point $p$ of $h$ that is not contained in $\mathrm{span}\left\{\left(\begin{smallmatrix}0\\0\\1\end{smallmatrix}\right)\right\}$, and $h$ is constant along $f(t):=e^{tC} p$, $t\in\mathbb{R}$. Since $c\ne 0$, it follows that $f$ is not periodic. The image of $f$ is contained in the affine plane $E:=p+\mathrm{span}\left\{\left(\begin{smallmatrix}1\\0\\0\end{smallmatrix}\right),\left(\begin{smallmatrix}0\\1\\0\end{smallmatrix}\right)\right\}$, and every line $L$ therein intersects the image of $f$ in countably infinitely many distinct points. Hence, by being a polynomial, $h$ must be constant along $L$, which is a contradiction to Lemma \ref{lem_no_rays}. We conclude that there exists no hyperbolic polynomial $h$ that is invariant under $\mathfrak{g}_2$.
			
			Next, we show that the scalar curvature of every special homogeneous surfaces $\mathcal{H}$ contained in the level set of $h=x^ay^bz^c$ as in \eqref{eqn_h_abc_form} with respect to the centro-affine fundamental form $g$ vanishes. Using a right-invariant, with respect to $g$ orthonormal, frame $X_1,X_2$ of $T\mathcal{H}$, we obtain from $\mathfrak{g}_1$ \eqref{eqn_g_1} being abelian and Koszul's formula that $\nabla^g_{X_i}X_j=0$ for all $i,j\in\{1,2\}$. Hence, the curvature of $(\mathcal{H},g)$ vanishes. Alternatively, we can consider a left-invariant orthonormal co-frame $\theta=\left(\begin{smallmatrix}e_1\\ e_2\end{smallmatrix}\right)$ of $T^*\mathcal{H}$, so that
				\begin{equation*}
					0=\!\D \theta=-\omega_{\mathrm{LC}}\wedge\theta,
				\end{equation*}
			where $\omega_{\mathrm{LC}}$ denotes the connection form of the Levi-Civita connection $\nabla^g$. The above equation follows from $\mathfrak{g_1}$ being abelian and, hence, $\D e_i(X,Y)=-e_i\left([X,Y]\right)=0$ for $i\in\{1,2\}$ and all $X,Y\in\mathfrak{X}(\mathcal{H})$. Thus, $\omega_{\mathrm{LC}}=0$, showing that the curvature form $\Omega_{\mathrm{LC}}=\D \omega_{\mathrm{LC}}+\omega_{\mathrm{LC}}\wedge\omega_{\mathrm{LC}}$ must also vanish identically.
			
			This finishes the proof of Theorem \ref{thm_main} and Proposition \ref{prop_scalar_curvature}.\qed
			
\subsection{Special homogeneous surfaces are singular at infinity}
	\begin{proof}[Proof of Proposition \ref{prop_shs_sing_at_inf}:]
		All connected components of the special homogeneous surfaces listed in Theorem \ref{thm_main} are isomorphic. Hence, it suffices to prove Proposition \ref{prop_shs_sing_at_inf} for any chosen connected component in each case.
		
		For Thm. \ref{thm_main} \eqref{thm_main_i}, $\D h$ vanishes identically on $\{h=0\}$. Hence, every connected component of $\mathrm{hyp}_1(h)$ is singular at infinity.
		
		The boundary of one of the connected components of $\mathrm{hyp}(h)$ in Thm. \ref{thm_main} \eqref{thm_main_ii} contains the ray $L=\mathbb{R}_{>0}\cdot(1,0,-1)^\mathrm{T}$. For any choice of $k$ with $1\leq k<\frac{\tau}{2}$ we find that $\D h$ vanishes identically along $L$. Hence, any connected component of $\mathrm{hyp}_1(h)$ is singular at infinity.
		
		Lastly, observe that in any case listed in Thm. \ref{thm_main} \eqref{thm_main_iii}, one connected component of $\mathrm{hyp}(h)$ is given by $\{x>0,\,y>0,\,z>0\}$, and its boundary contains the positive rays in the coordinate axis. We now check that that e.g. along $\{x>0,\,y=0,\,z=0\}$, $\D h$ vanishes identically. This shows that also in this case, every special homogeneous surface is singular at infinity.
	\end{proof}


\begin{thebibliography}{ABCD}

\bibitem[BH]{BH} P. Br\"an\'en, J. Huh, \textit{Lorentzian polynomials}, Ann. of Math. \textbf{192}, 821--891 (2020).

\bibitem[CDL]{CDL} V. Cort\'es, M. Dyckmanns, and D. Lindemann, \textit{Classification of complete projective special real surfaces}, Proc. London Math. Soc. \textbf{109}, No. 2, 423--445 (2014).

\bibitem[CDJL]{CDJL} V. Cort\'es, M. Dyckmanns, M. J\"ungling, and D. Lindemann, \textit{A class of cubic hypersurfaces and quaternionic K\"ahler manifolds of co-homogeneity one}, Asian J. Math., Vol. \textbf{25}, No. 1, 1--30 (2021).

\bibitem[CHM]{CHM} V. Cort\'es, X. Han, and T. Mohaupt, {\it Completeness in supergravity constructions}, Commun. Math. Phys. {\bf 311}, No. 1, 191--213 (2012).

\bibitem[CNS]{CNS} V. Cort\'es, M. Nardmann, and S. Suhr, {\it Completeness of hyperbolic centroaffine hypersurfaces}, Comm. Anal. Geom., Vol. \textbf{24}, No. 1, 59--92 (2016).

\bibitem[DV]{DV} B. de Wit, A. Van Proeyen, {\it Special geometry, cubic polynomials and homogeneous quaternionic spaces}, Comm. Math. Phys. {\bf 149}, No. 2, 307--333 (1992).

\bibitem[G]{G} L. G{\aa}rding, \textit{An Inequality for Hyperbolic Polynomials}, Journal of Mathematics and Mechanics, Vol. \textbf{8}, No. 6, 957--965 (1959).

\bibitem[H]{H} L. H\"ormander, \textit{Notions of Convexity}, Birkh\"auser Boston, MA (1994).

\bibitem[L1]{L1} D. Lindemann, \textit{Structure of the class of projective special real manifolds and their generalisations}, PhD-thesis (2018).

\bibitem[L2]{L2} D. Lindemann, \textit{Properties of the moduli set of complete connected projective special real manifolds}, Math. Z. \textbf{303}(2) (2023).

\bibitem[L3]{L3} D. Lindemann, \textit{Limit geometry of complete projective special real manifolds},\hfill\textcolor{white}{.} \href{https://arxiv.org/abs/2009.12956}{arXiv:2009.12956}.

\bibitem[L4]{L4} D. Lindemann, \textit{Special geometry of quartic curves}, \href{https://arxiv.org/abs/2206.12524}{arxiv:2206.12524}.

\bibitem[L5]{L5} D. Lindemann, \textit{Special homogeneous curves}, \href{https://arxiv.org/abs/2208.06890}{arxiv:2208.06890}.

\bibitem[LSZH]{LSZH} A.-M. Li, U. Simon, G. Zhao, Z. J. Hu, \textit{Global Affine Differential Geometry of Hypersurfaces}, 2nd edn., W. de Gruyter, Berlin (2015).

\bibitem[MS]{MS} O. Macia, A. Swann, \textit{Twist Geometry of the c-Map}, Comm. Math. Phys. \textbf{336}, 1329--1357 (2015).

\bibitem[T]{T} B. Totaro, \textit{The curvature of a Hessian metric}, Int. J. Math. \textbf{15}, 369--391 (2004).

\bibitem[W]{W} P.M.H. Wilson, \textit{Sectional curvatures of K\"ahler moduli}, Math. Ann. \textbf{330}, 631--664 (2004).

\end{thebibliography}
\end{document}